\documentclass[runningheads, a4paper]{amsart}
\usepackage{amsmath,amssymb,amsthm}
\usepackage{amsaddr}
\usepackage{xspace}
\usepackage[disable]{todonotes}
\usepackage{virginialake}
\usepackage[colorlinks]{hyperref}
\usepackage[capitalise]{cleveref}
\usepackage{tikz}
\usetikzlibrary {arrows.meta,automata,positioning}
\usetikzlibrary{arrows,cd}
\usepackage{mathtools}
\usepackage{multirow}
\begin{document}


\newtheorem{theorem}{Theorem}
\newtheorem{conjecture}{Conjecture}
\newtheorem{lemma}[theorem]{Lemma}
\newtheorem{proposition}[theorem]{Proposition}
\newtheorem{observation}[theorem]{Observation}
\crefformat{observation}{Observation~#2#1#3}

\newtheorem{corollary}[theorem]{Corollary}
\newtheorem{fact}[theorem]{Fact}

\theoremstyle{definition}
\newtheorem{definition}[theorem]{Definition}
\newtheorem{example}[theorem]{Example}
\newtheorem{remark}[theorem]{Remark}
\newtheorem{convention}[theorem]{Convention}

\renewcommand{\phi}{\varphi}
\renewcommand{\emptyset}{\varnothing}
\renewcommand{\epsilon}{\varepsilon}
\newcommand{\defname}[1]{\textbf{{#1}}}
\newcommand{\df}{:=}
\newcommand{\bnf}{::=}
\newcommand{\Pow}{\mathcal P}
\newcommand{\pow}[1]{\Pow(#1)}

\newcommand{\IH}{\mathit{IH}}

\newcommand{\Nat}{\mathbb{N}}
\newcommand{\Int}{\mathbb{Z}}
\newcommand{\Cal}[1]{\mathcal{#1}}
\newcommand{\SF}[1]{\mathsf{#1}}
\newcommand{\BB}[1]{\mathbb{#1}}
\newcommand{\resp}{\textit{resp.}\xspace}
\newcommand{\aka}{\textit{aka}\xspace}
\newcommand{\viz}{\textit{viz.}\xspace}
\newcommand{\ie}{\textit{i.e.}\xspace}
\newcommand{\cf}{\textit{cf.}\xspace}
\newcommand{\wrt}{\textit{wrt}\xspace}
\newcommand{\rk}[1]{\SF{rk}(#1)}

\newcommand{\red}[1]{{\color{red} #1}}
\newcommand{\blue}[1]{{\color{blue} #1}}
\newcommand{\mgt}[1]{\textcolor{magenta}{#1}}
\newcommand{\cyn}[1]{\textcolor{cyan}{#1}}
\newcommand{\olv}[1]{\textcolor{olive}{#1}}
\newcommand{\orange}[1]{{\color{orange}#1}}
\newcommand{\purple}[1]{{\color{purple}#1}}

\newcommand{\anupam}[1]{\todo[inline]{Anu: #1}}
\newcommand{\abhishek}[1]{\todo[inline]{Abhi: #1}}

\newcommand{\Alphabet}{\mathcal{A}}
\newcommand{\Var}{\mathcal{V}}

\newcommand{\proves}{\vdash}

\newcommand{\Lang}{\mathcal L}
\newcommand{\lang}[1]{\Lang(#1)}

\newcommand{\wLang}{\Lang}
\newcommand{\wlang}[1]{\wLang(#1)}

\newcommand{\FV}{\mathrm{FV}}
\newcommand{\fv}[1]{\FV(#1)}

\newcommand{\fint}[1]{\lceil #1 \rceil}
\newcommand{\flred}{\rightarrow_{\FL}}
\newcommand{\flredeq}{\flred^{=}}
\newcommand{\fl}[1]{\FL(#1)}
\newcommand{\eqfl}{=_\FL}
\newcommand{\lefl}{<_\FL}
\newcommand{\leqfl}{\leq_\FL}
\newcommand{\geqfl}{\geq_\FL}

\newcommand{\subform}{\sqsubseteq}
\newcommand{\supform}{\sqsupseteq}
\newcommand{\dle}{\prec}
\newcommand{\dleq}{\preceq}
\newcommand{\dge}{\succ}
\newcommand{\dgeq}{\succeq}
\newcommand{\id}{\SF{id}}
\newcommand{\FL}{\mathrm{FL}}


\newcommand{\KA}{\mathsf{KA}}
\newcommand{\lhKA}{\ell\KA}

\newcommand{\RLA}{\ensuremath{\mathsf{RLA}}\xspace}
\newcommand{\RLL}{\ensuremath{\mathsf{RLL}}\xspace}
\newcommand{\fd}[1]{#1^{\mathit{fd}}}
\newcommand{\RLLfd}{\fd\RLL}

\newcommand{\RLLLang}{\RLL_\Lang}

\newcommand{\ubdd}{\mathsf u}
\newcommand{\gdd}{\mathsf g}
\newcommand{\gRLL}{\gdd\RLL}
\newcommand{\gRLLfd}{\fd\gRLL}
\newcommand{\uRLL}{\ubdd\RLL}
\newcommand{\uRLLfd}{\fd\uRLL}

\newcommand{\LRLLLanghat}{\mathsf L\widehat\RLL_\Lang}
\newcommand{\CRLLLang}{\CRLL_\Lang}
\newcommand{\LRLLLang}{\mathsf L \RLL_\Lang}

\newcommand{\ind}{\mathsf{ind}}
\newcommand{\coind}{\mathsf{coind}}

\newcommand{\CRLLLangc}{\CRLLLang^c}
\newcommand{\LRLLLangc}{\LRLLLang^c}


\renewcommand{\c}[1]{#1^c}
\newcommand{\infrule}{\mathsf{r}}

\newcommand{\strat}[1]{\mathfrak{#1}}

\newcommand{\gtop}{\top_\Alphabet}


\newcommand{\Eloise}{\exists}
\newcommand{\Abelard}{\forall}

\newcommand{\prover}{\mathbf{P}}
\newcommand{\denier}{\mathbf{D}}

\newcommand{\cantor}{\mathcal C}


\newcommand{\LTL}{\mathsf{LTL}}
\newcommand{\muLTL}{\mu\LTL}
\newcommand{\nxt}{\bigcirc}
\newcommand{\U}{\mathbf{U}}
\newcommand{\props}{\mathbf P}
\newcommand{\limp}{\rightarrow}
\newcommand{\liff}{\leftrightarrow}

\newcommand{\interp}[2]{#2^{#1}}
\newcommand{\ef}[1]{#1^{\circ}}
\newcommand{\fe}[1]{#1^{\bullet}}

\title[An algebraic theory of $\omega$-regular languages]{An algebraic theory of $\omega$-regular languages, via $\mu\nu$-expressions}

\author{Anupam Das \and Abhishek De}
\address{School of Computer Science, University of Birmingham, UK}
\date{\today}

\begin{abstract}
Alternating parity automata (APAs) provide a robust formalism for modelling infinite behaviours and play a central role in formal verification. Despite their widespread use, the algebraic theory underlying APAs has remained largely unexplored. In recent work~\cite{DD24a}, a notation for non-deterministic finite automata (NFAs) was introduced, along with a sound and complete axiomatisation of their equational theory via \emph{right-linear algebras}. In this paper, we extend that line of work, in particular to the setting of infinite words. We present a dualised syntax, yielding a notation for APAs based on \emph{right-linear lattice} expressions, and provide a natural axiomatisation of their equational theory with respect to the standard language model of $\omega$-regular languages. The design of this axiomatisation is guided by the theory of fixed point logics; in fact,  the completeness factors cleanly through the completeness of the linear-time $\mu$-calculus.
\end{abstract}

\maketitle    

\section{Introduction}
\label{sec:introduction}

\subsection{A half century of \texorpdfstring{$\omega$}{}-automata theory}

$\Omega$-automata, \ie finite state machines running on infinite inputs, are useful for modelling behaviour of systems that are not expected to terminate, such as hardware, operating systems and control systems. The prototypical $\omega$-automaton model, \emph{B\"uchi automaton}, is widely used in model checking~\cite{VW94,GPVW96,GO01,Holzmann11}.   

The theory of $\omega$-regular languages, \ie languages accepted by $\omega$-automata, have been studied for more than half a century. B\"uchi's famous complementation theorem~\cite{Büchi90} for his automata is the engine underlying his proof of the decidability of monadic second-order logic (MSOL) over infinite words. Its extension to infinite trees, \emph{Rabin's Tree Theorem}~\cite{Rabin68}, is often referred to as the `mother of all decidability results'.

McNaughton~\cite{McNaughton66} showed that, while B\"uchi automata could not be determinised per se, a naturally larger class of acceptance conditions (Muller or parity) allowed such determinisation, a highly technical result later improved by Safra~\cite{Safra}. 
A later relaxation was the symmetrisation of the transition relation itself: instead of only allowing non-deterministic states, allow co-nondeterministic ones too. 
This has led to beautiful accounts of $\omega$-regular language theory via the theory of positional and finite memory games. 
The resulting computational model, \emph{alternating parity automaton} (APA), 
is now the go-to model in textbook presentations, e.g.~\cite{GTW03}. 
Indeed, their features more closely mimic those of logical settings where such symmetries abound, e.g. linear-time $\mu$-calculus~\cite{Vardi1996} and MSOL over infinite words.

\subsection{An algebraic approach}

In the finite world, the theory of regular languages have been axiomatised as \emph{Kleene Algebras} ($\KA$s). In fact, $\KA$s are part of a bigger cohort of \emph{regular algebras}  and they have been studied for several decades and completeness proofs for different variants have been obtained~\cite{Salomaa66,Krob91,Kozen94,Boffa90,Boffa95}. $\KA$s and various extensions have found applications in specification and verification of programs and networks~\cite{AFGJKSW14}. 

However, note $\KA$s and other regular algebras axiomatise the equational theory of \emph{regular expressions} as opposed to NFAs. Although they are equi-expressive, regular expressions are not quite a `notation' for NFAs. Nonetheless, NFAs may be given a bona fide notation by identifying them with \emph{right-linear grammars}.  Recall that a right-linear grammar is a CFG where each production has RHS either $aX$ or $\epsilon$. They may also be written as \emph{right-linear expressions}, by choosing an order for resolving non-terminals. Formally, \defname{right-linear expressions} (aka \defname{RLA expressions}), written $e,f,\dots$, are generated by:
\[
e,f,\dots \quad ::= \quad 
1 \quad \mid \quad 
 X \quad \mid \quad
e + f \quad \mid \quad a\cdot e \quad \mid \quad \mu X e
\]
for $a\in\Alphabet$, a finite \defname{alphabet} and $X \in \Var$, a countable set of \defname{variables}. Indeed~\cite{DD24a} takes this viewpoint seriously and proposed an alternative algebraic foundation of regular language theory, via \emph{right-linear algebras} ($\RLA$s). Notably, $\RLA$s are strictly more general than $\KA$s, as they lack any multiplicative structure. In particular, this means that $\omega$-languages naturally form a model of them (unlike $\KA$s). This is the starting point of the current work.

In this work, we investigate the algebraic structures induced by the theory of APAs. To do so, we dualise the (1-free)\footnote{This restriction imposed to so that the intended interpretation is just sets of $\omega$-words not $\le\omega$-words.} syntax of RLA expressions to obtain \emph{right-linear lattice} (RLL) expressions, formally generated by:
\[e,f,\dots \quad\bnf\quad   X \quad\mid\quad a\cdot e \quad\mid\quad  e+f \quad\mid\quad \mgt{e \cap f} \quad\mid\quad \mu X e(X) \quad\mid\quad \mgt{\nu X e(X)}\]    

Compared to RLA expressions, RLL expressions enjoy more symmetric relationships to games and consequently, are a notation for APA. Our main contribution is a sound and complete axiomatisation $\RLLLang$ of the theory of RLL expressions for the language model. 

\subsection{Roadmap}

In~\Cref{sec:prel}, we recall right-linear algebras and define RLL expressions, a notation for APAs. We identify several principles governing their behaviour in the standard model $\Lang$ of $\omega$-languages; namely, their interpretations satisfy a theory of bounded distributive lattices, certain lattice homomorphisms and least and greatest fixed points (of definable operators). To motivate the final axiomatisation in~\Cref{sec:axiomatisation}, we first syntactically recover complements in~\Cref{sec:complement}. In~\Cref{sec:completeness}, we prove the completeness of the axiomatisation by reducing it to the completeness of linear time $\mu$-calculus. We conclude with some remarks on the axiomatisation and comparison with existing literature in~\Cref{sec:conclusion}. For the sake of self-containment, some (now standard) results of cyclic proof theory are given in \cref{sec:eval-game}.  

\subsection{Related work} 

Two kinds of variations of $\KA$s are relevant to this work. Firstly, the generalisation of regular algebras to $\omega$-regular algebras~\cite{Wagner76,Cohen00,LS12,CLS15}, by axiomatising the theory of $\omega$-regular expressions, a generalisation of regular expressions admitting terms of the form $e^\omega$, for $e$ an $\omega$-regular expression. Secondly, following the idea of dualisation, dualising every binary operation in $\KA$s leads to \emph{action lattices}, an extension with meet (dual to the sum), and residuals (adjoint to the product). Since $\RLA$s do not have products, we do not need residuals in its dualisation -- so, perhaps, \emph{Kleene lattices}~\cite{Brunet17,DasPous18}, the extension of $\KA$s with meet is the closest cousin of our proposed right-linear lattices. 
\section{Right-linear lattice expressions for \texorpdfstring{$\omega$}{}-regular languages}
\label{sec:prel}

Let us fix a finite set $\Alphabet$ (the \defname{alphabet}) of \defname{letters}, written $a,b,$ etc., and a countable set $\Var$ of \defname{variables}, written $X,Y,$ etc.

\subsection{RLL expressions and \texorpdfstring{$\omega$}{}-regular languages}

Recall that RLL expressions, written $e,f,\dots$, are generated by:
\[e,f,\dots \quad\bnf\quad   X \quad\mid\quad a\cdot e \quad\mid\quad  e+f \quad\mid\quad e \cap f \quad\mid\quad \mu X e(X) \quad\mid\quad \nu X e(X)\]  

\noindent for $a\in\Alphabet$ and $X \in \Var$. We usually just write $ae$ instead of $a\cdot e$. A variable $X$ is said to occur \defname{freely} in an expression $e$ if it not under the scope of any binder $\mu X$ or $\nu X$. An expression is said to be \defname{closed} if it has no occurrences of free variables.  

\begin{remark}[0]
\label{rem:0}
The original presentation of right-linear expressions includes a symbol $0$ that was always interpreted as a unit for $+$ in structures over this syntax. Here we shall more simply just write $0\df \mu X X$, and remark on the consequences of this choice as we go.
\end{remark}

\noindent The intended interpretation of an RLL expression is a language of $\omega$-words over $\Alphabet$.

\begin{definition}[Interpretation]
Let us temporarily expand the syntax of RLL expressions to include each language $A\subseteq \Alphabet^\omega$ as a constant symbol. We interpret each closed expression (of this expanded language) as a subset of $ \Alphabet^\omega$ as follows:
\begin{itemize}
    \item $\lang A \df A$
    \item $\lang {e+f} \df \lang e \cup \lang f$
    \item $\lang{e\cap f} \df \lang e \cap \lang f$
    \item $\lang {ae} \df \{a \sigma \mid  \sigma \in \lang e\}$
    \item $\lang {\mu X e(X)} \df \bigcap_{A\subseteq\Alphabet^\omega} \{A \mid A \supseteq \lang {e(A)} \}$
    \item $\lang {\nu X e(X)} \df \bigcap_{A\subseteq\Alphabet^\omega} \{A \mid A  \subseteq \lang {e(A)} \} $
\end{itemize}
\end{definition}

\noindent Note that \cref{rem:0} is justified by this interpretation: indeed $\lang {\mu X X}$ is just the empty language.

\begin{remark}
    [$\top$]
    Dual to $0\df \mu X X$, we define $\top \df \nu X X$, that denotes the universal language in $\Lang$.
\end{remark}

To justify that $\mu$ and $\nu$ are indeed interpreted as \emph{fixed point} operators, we will first recall some terminology. Let $(S,\leqslant_S)$ be a complete lattice. Then, $x\in S$ is said to be a \defname{prefixed} (\defname{postfixed} respectively) point of a morphism $f:S\to S$ if $f(x)\leqslant_S x$ ($x\leqslant_S f(x)$ respectively). If $x$ is both a pre and postfixed point, it is called a \defname{fixed point} of $f$.  

\begin{theorem}[Knaster-Tarski theorem~\cite{Knaster-Tarski,Tarski1955ALF}]
\label{thm:Tarski}
Let $f:S\to S$ be a monotonic function. The set of fixed points of $f$ is non-empty and equipped with $\leqslant_S$ forms a complete lattice.
\end{theorem}

Let us now point out that $\pow{\Alphabet^{\omega}}$ indeed forms a complete lattice under $\subseteq$, and closed under concatenation with letters on the left. Since all the operations are monotone, $\lang{\mu X e(X)}$ and $\lang{\nu X e(X)}$ are indeed the least and greatest fixed point of the operation $A\mapsto \lang{e(A)}$, by the Knaster-Tarski theorem.

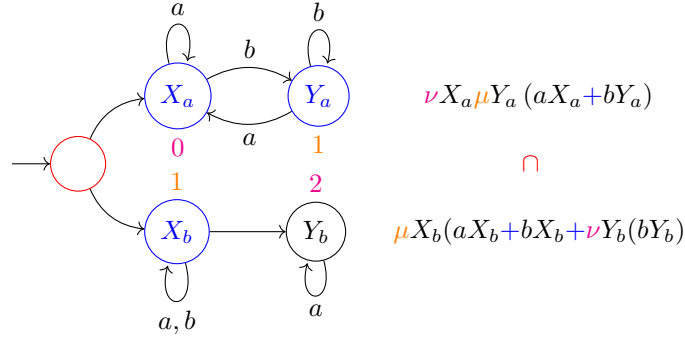
\begin{figure}[t]
\centering
\begin{tikzpicture}[baseline=(root)]
    \node[circle, draw, color=red] (root) at (0,0) {$\phantom{\cap\  }$ };
    \node[right=.5cm of root, yshift=.9cm, circle, draw, color=blue, label=below:{$\mgt 0$}] (Xa) {$X_a$};
    \node[right=1cm of Xa, circle, draw, color=blue,label=below:{$\orange 1$}] (Ya) {$Y_a$};
        \node[left=.5cm of root] (nil) {};
    \node[right=.5cm of Ya] (lin) {};
    \draw[->] (nil) to (root);
    
    \node[right=.5cm of root, yshift=-.9cm, circle, draw, color=blue,label=above:{$\orange 1$}] (Xb) {$X_b$};
    \node[right=1cm of Xb, circle, draw,label=above:{$\mgt 2$}] (Yb) {$Y_b$};
    
    \draw[->] (root) edge[bend left] (Xa);
    \draw[->] (Xa) edge[bend left] node[above] {$b$} (Ya);
    \draw[->] (Ya) edge[bend left] node[below] {$a$} (Xa);
    \draw (Xa) edge[loop above] node[above] {$a$} (Xa);
    \draw (Ya) edge[loop above] node[above] {$b$} (Ya);

    \draw[->] (root) edge[bend right] (Xb);
    \draw[->] (Xb) edge (Yb);
    \draw (Xb) edge[loop below] node[below] {$a,b$} (Xb);
    \draw (Yb) edge[loop below] node[below] {$a$} (Yb);



    \node[right=.5cm of Ya] {$\ \ \ \mgt \nu X_a \orange \mu Y_a \left(aX_a \blue + bY_a\right)$};
    \node[right=.5cm of Ya, yshift=-.9cm] {$\ \ \ \ \ \ \ \ \ \ \ \ \ \  \red \cap $};
    \node[right=.5cm of Yb] {$\orange \mu X_b (aX_b \blue + bX_b \blue + \mgt \nu Y_b (bY_b)$};
\end{tikzpicture}
\caption{The alternating parity automata $\mathbf A_{i_a\cap f_b}$.}
\label{fig:apa-eg}
\end{figure}

\begin{example}
\label{eg:exprs}
Let us consider some examples of RLL expressions and the languages they compute in $\Lang$, over the alphabet:
\begin{itemize}
    \item $i_a \df \nu X \mu Y (aX + bY)$ computes the language $I_a$ of words with infinitely many $a$s:
    \begin{itemize}
        \item First note that, for any language $A$, we have that $\mu Y (A + bY) $ computes $b^*A$.
        \item Now let us show that $I_a$ is a postfixed point of $X \mapsto \mu Y (aX + bY) $. By the above point, it suffices to show that $I_a \subseteq b^*aI_a$, which holds since every word $w$ with infinitely many $a$s can be written $w = b^*a w'$.
        \item Now suppose $B$ is another postfixed point, i.e.\ that $B\subseteq b^*aB$. Then we have $B\subseteq b^*aB \subseteq b^*a b^*a B \subseteq \cdots \subseteq (b^*a)^\omega = I_a$.
    \end{itemize}
    \item $f_b \df \mu X (bX + aX +  a\nu YaY)$ computes the language $F_b$ of words with at most finitely many $b$s:
    \begin{itemize}
        \item First note that, $\nu YaY$ computes $a^\omega$.
        \item By a similar argument as above, $F_b$ is a prefixed point of $X \mapsto bX+aX+aa^\omega$.
    \end{itemize}
    \item $i_a \cap f_b$ computes the language $I_a\cap F_b$ of words with infinitely many $a$s and at most finitely many $b$s. Note that over $\Alphabet=\{a,b\}$, $I_a\cap F_b=F_b$ but not in general (say when $\Alphabet=\{a,b,c\}$).
\end{itemize}    
\end{example}

As the readers might have expected, the range of $\lang\cdot$ is just the $\omega$-regular languages.

\begin{proposition}
\label{prop:omega-reg-have-munu-exps}
A language $L\subseteq\Alphabet^\omega$ is $\omega$-regular if and only if there is an RLL expression $e$ such that $\lang e=L$.    
\end{proposition}

One direction, exhaustion of all $\omega$-regular languages, follows swiftly from the inductive definition of the set of all $\omega$-regular languages and was established in previous work~\cite{DD24a}, without making use of $\cap$. To prove the converse, we will define an APA $\mathbf A_e$ for each expression $e$ such that $\lang e = \lang{\mathbf A_e}$.

\begin{definition}
[Fischer-Ladner]
\label{def:fl}
Define $\flred$ as the smallest relation on expressions satisfying:
\begin{itemize}
    \item $ae \flred e$.
    \item $e_0\star e_1 \flred e_i$, for $i\in \{0,1\}$ and $\star \in \{+,\cap\}$.
    \item $\sigma X e(X) \flred e(\sigma Xe(X))$, for $\sigma \in \{\mu,\nu\}$.
\end{itemize}
Write $\leqfl$ for the reflexive transitive closure of $\flred$.
The \defname{Fischer-Ladner ($\FL$) closure} of an expression $e$, written $\fl e$, is $\{f\leqfl e\}$.
We also write $e\subform f$ if $e$ is a subformula of $f$, in the usual sense.
\end{definition}

It is well-known that $\fl e $ is always finite. 
This follows by induction on the structure of $e$, relying on the equality $\fl{\sigma X e} = \{\sigma X e\} \cup \{f [ \sigma X e / X] : f \in \fl e\}$ (see, e.g., \cite{DD24a} for further details).

From here we can readily define $\mathbf A_e$ with:
\begin{itemize}
    \item \emph{States}:  $\fl e$, with expressions $0,f+g$ existential and expressions $\top,f\cap g$ universal.\footnote{Again, it does not matter whether other expressions are existential or universal states, as there is a unique instance of $\flred$ from them.} The initial state is $e$.
    \item \emph{Transitions}: 
    \begin{itemize}
        \item $af \underset a \to f$ whenever $af \in \fl e$; and,
        \item $g \to g'$ whenever $g\flred g'$ and $g$ is not of form $af$.
    \end{itemize} 
    \item \emph{Colouring}: any function $c_e: \fl e \to \Nat $ s.t.:
    \begin{itemize}
        \item $c_e$ is monotone wrt subformulas, i.e.\ if $f\subform g \implies c(f) \leq c(g)$; and,
        \item $c_e$ assigns $\mu$ and $\nu$ formulas odd and even numbers, respectively, i.e.\ always $c_e(\mu Xf(X))$ is odd and $c_e(\nu X f(X))$ is even.
    \end{itemize}
\end{itemize}

\begin{theorem}
\label{thm:lang-a-e}
For every $e$, $\lang e=\lang{\mathbf A_e}$.
\end{theorem}

Note that \cref{prop:omega-reg-have-munu-exps} follows from \cref{thm:lang-a-e}. To prove \Cref{thm:lang-a-e}, we will introduce a game-theoretic mechanism for deciding word membership in $\lang e$. This was introduced in~\cite{DD24a} without $\cap$ and can be straightforwardly lifted to our setting (see \cref{sec:eval-game}). We will simply illustrate $\mathbf A_e$ with an example and move on.

\begin{example} Consider $i_a\cap f_b$ as defined in~\cref{eg:exprs}. If we follow the construction above, we have the APA in~\cref{fig:apa-eg} where blue states are existential, red states are universal, magenta is an even colour, and orange is an odd colour. 

Let us check that $\lang{\mathbf A_{i_a\cap f_b}}$ is the set of all words with infinitely many $a$s and finitely many $b$s. Let $w$ be such a word. Then, a path in the run tree visits $X_a$ infinitely often or loops on $Y_b$. In both cases, it is accepting. Now suppose $w$ is a word that contains infinitely many $b$s. Then, its run tree contains a path that loops on $X_a$. This path is not accepting. Similarly, $w$ is a word containing finitely many $a$s then its run tree contains the bad path looping on $Y_a$. 

%
%
%
\end{example}

\begin{remark}[A subtlety about $\varepsilon$]
Note that we have allowed $\varepsilon$-transitions in our APAs in order to mimic the RLL syntax as closely as possible. Let us point out that our APAs indeed still only compute the $\omega$-regular languages.
\end{remark}

\subsection{Some properties of the intended model}
\label{sec:props-of-lang}

Let us take a moment to remark upon some principles valid in the intended interpretation $\Lang$ of RLL expressions, in order to motivate the axiomatisation we introduce later. 
As usual we write $e\leq f \df e+f =f$, equivalently $e = e\cap f$ (so in $\Lang$, $\leq$ just means $\subseteq)$.
First:

\begin{itemize}
    \item $(0,\top, +,\cap)$ forms a bounded distributive lattice:\footnote{Some of these axioms are redundant, but we include them all to facilitate the exposition.}  
    \begin{equation}
        \label{eq:dist-lattice}
        \begin{array}{r@{\ = \ }l}
        e + 0 & e \\
        e + (f + g) & (e + f ) + g\\
        e + f & f + e \\
        e+e & e \\
        e + (e\cap f) & e\\
        e + (f\cap g) & (e+ f)\cap(e+ g)
    \end{array}
    \qquad
     \begin{array}{r@{\ = \ }l}
         e \cap \top & e\\
         e \cap (f \cap g) & (e \cap f ) \cap g\\
         e \cap f & f \cap e\\
         e \cap e & e \\
         e \cap (e + f) & e \\
         e \cap (f+ g) & (e\cap f)+(e\cap g)
    \end{array}
    \end{equation}
\end{itemize}


\begin{itemize}
    \item Each $a\in \Alphabet$ is a (lower) semibounded lattice homomorphism:
    \begin{equation}
        \label{eq:letter-lbdd-lattice-homo}
         \begin{array}{r@{\ = \ }l}
     a0 & 0 \\
     a(e+f) & ae + af\\
    a(e\cap f) & ae \cap af
    \end{array}
    \end{equation}
\end{itemize}

In particular, of course $\Lang \not \models a\top = \top$, so in this sense $0$ and $\top$ do not behave dually in $\Lang$. 
Instead we have a variant of this principle, indicating that the homomorphisms freely factor the structure:

\begin{itemize}
    \item The ranges of $a\in\Alphabet$ partition the domain:
    \begin{equation}
    \label{eq:letters-freely-cogenerate-structure}
        \begin{array}{r@{\ = \ }ll}
         ae \cap bf & 0 & \text{whenever $a\neq b$} \\
         \top & \sum\limits_{a\in \Alphabet} a\top
    \end{array}
    \end{equation}
\end{itemize}

Finally, $\Lang $ is a complete lattice and so interprets the least and greatest fixed points as such. 
Being a complete lattice is a \emph{second-order} property, but we have the following first order (even quasi-equational) consequences:
\begin{itemize}
\item $\mu X e(X)$ is a least prefixed point of $X\mapsto e(X)$:
\begin{equation}
    \label{eq:mu-is-least-prefix}
    \begin{array}{ll}
        \text{(Prefix)} & e(\mu X e(X)) \leq \mu X e(X)  \\
         \text{(Induction)} & e(f) \leq f \implies \mu X e(X) \leq f 
    \end{array}
\end{equation}
    \item $\nu X e(X)$ is a greatest postfixed point of $X\mapsto e(X)$:
    \begin{equation}
        \label{eq:nu-is-greatest-postfix}
        \begin{array}{ll}
         \text{(Postfix)} & \nu X e(X)\leq e(\nu X e(X)) \\
         \text{(Coinduction)} & f\leq e (f) \implies f\leq \nu X e(X)
    \end{array}
    \end{equation}
Note that Induction and Coinduction are axiom \emph{schemas}. In fact, it is quite standard that first order axiomatisation of (Co)Induction presented as schema (\cf Peano Arithmetic). 
\end{itemize}

\begin{example}
[0]
\label{ex:0-is-unit}
Recall $0:=\mu XX$ and $\top:=\nu XX$. Indeed $0\leq e$ (i.e.\ $0+e=e$) is a consequence of the axioms \eqref{eq:mu-is-least-prefix} above: it follows by Induction from $e\leq e$. Dually $e\leq \top$ follows from \eqref{eq:nu-is-greatest-postfix}.
\end{example}

Recall that RLA expressions are notation for NFAs and thus can be duly interpreted as regular languages over \emph{finite} words. In previous work~\cite{DD24a}, soundness and completeness of a subset of the above mentioned axioms for RLA expressions with respect to the language interpretation (also written $\Lang$ hedging the risk of confusion). Writing $\RLA$ for the subset of axioms  from \cref{eq:dist-lattice,eq:letter-lbdd-lattice-homo,eq:letters-freely-cogenerate-structure,eq:mu-is-least-prefix,eq:nu-is-greatest-postfix} not involving $\cap, \top, \nu$, we have:

\begin{theorem}
[\cite{DD24a}]
\label{thm:RLA-sound-complete}
For RLA expressions $e,f$, $\RLA \proves e=f$ $\iff$ $\lang{e}=\lang{f}$.
\end{theorem}

The goal of the present work is to establish a similar sort of result for RLL expressions, in the $\omega$-regular world rather than the (finitely) regular world. 
\section{Boolean subalgebra of RLL expressions}
\label{sec:complement}

As the $\omega$-regular languages are closed under complementation, we actually have that the initial term submodel of RLL expressions in $\Lang$ forms a Boolean algebra.
In this section, we shall inline this structure axiomatically.

\subsection{Complements}
We can define complements of the RLL expressions, wrt $\Lang$, quite simply, thanks to closure of the syntax under duality:
\begin{definition}
    [Complement]
    Define $\c e $ by induction on an expression $e$:
    \begin{itemize}
        \item $\c {(ae)} \df a\c e + \sum\limits_{b\neq a} b\top$
        \item $\c X \df X$
        \item $\c{(e+f)} \df \c e \cap \c f$
        \item $\c{(e\cap f)} \df \c e + \c f$
        \item $\c{(\mu X e)} \df \nu X \c e$
        \item $\c{(\nu X e)} \df \mu X \c e $
    \end{itemize}
\end{definition}

\begin{proposition}
$e$ and $\c e $ are complementary in $\Lang$, i.e.\ $\lang {\c e } = \Alphabet^\omega \setminus \lang e$ for any closed expression $e$.      
\end{proposition}

\begin{proof}
In order to prove by induction, we will strengthen the statement. Let $e(X_1,\dots,X_n)$ be an RLL expression  with free variables $X_1,\dots,X_n$. We claim $\lang{e(A_1,\dots,A_n)^c}=\Alphabet^\omega\setminus \lang{e^c(\Alphabet^\omega\setminus A_1^c,\dots,\Alphabet^\omega\setminus A_n)}$ where $A_1,\dots,A_n$ are arbitrary languages over $\omega$-words. Now we induct on $e$.

\begin{itemize}
    \item Suppose $e=X$ then it is immediate. 
    \item Suppose $e=af$. Then
    \begin{align*}
        \lang{e^c}&= \lang{af^c+\sum_{b\ne a}b\top}\\
                  &= a\lang{f^c}\cup\bigcup_{b\ne a}b\Alphabet^\omega &&[\text{Definition of }\Lang]\\
                  &= a(\Alphabet^\omega\setminus\lang{f})\cup\bigcup_{b\ne a}b\Alphabet^\omega &&[\text{Hypothesis}]\\
                  &=(a\Alphabet^\omega\setminus a\lang{f})\cup\bigcup_{b\ne a}b\Alphabet^\omega\\
                  &=\Alphabet\Alphabet^\omega\setminus\lang{af}\\
                  &=\Alphabet^\omega\setminus\lang{e}&&[\because\Alphabet\Alphabet^\omega=\Alphabet^\omega]
    \end{align*}
    \item When $e=f+g$ or $e=f\cap g$, it is simple De Morgan reasoning.
    \item Suppose $e=\mu Xf(X)$. Then
    \begin{align*}
        \lang{e^c}&=\lang{\nu Xf(X)^c}\\
                  &=\bigcup_{A\subseteq\Alphabet^\omega}\{A\mid A\subseteq\lang{f(A)^c}\}\\
                  &=\bigcup_{A\subseteq\Alphabet^\omega}\{A\mid A\subseteq\Alphabet^\omega\setminus\lang{f(\Alphabet^\omega\setminus A)}\}&&[\text{Hypothesis}]\\
                  &=\bigcup_{A\subseteq\Alphabet^\omega}\{A\mid \lang{f(\Alphabet^\omega\setminus A)}\subseteq\Alphabet^\omega\setminus A\}\\
                  &=\Alphabet^\omega\setminus\bigcap_{A\subseteq\Alphabet^\omega}\{\Alphabet^\omega\setminus A\mid \lang{f(\Alphabet^\omega\setminus A)}\subseteq\Alphabet^\omega\setminus A\}
    \end{align*}
    \item The case when $e=\nu X f(X)$ is symmetric. \qedhere
\end{itemize}

\end{proof}

Thus the set of RLL expressions denote a Boolean subalgebra of $\Lang$, a fact subsumed by adequacy for $\omega$-regular languages, \cref{prop:omega-reg-have-munu-exps}.   Of course duality of $+,\cap$ hold in any bounded distributive lattice. The homomorphism axioms also guarantee that our definition of $\c{(ae)}$ is well-behaved:

    \begin{example}
    \label{ex:compl-of-action-in-lang}
        Let $\mathfrak L$ be a bounded distributive lattice (i.e.\ a model of \eqref{eq:dist-lattice}) satisfying \cref{eq:letter-lbdd-lattice-homo,eq:letters-freely-cogenerate-structure}, and suppose $A$ has a complement $\c A$ in $\mathfrak L$.\footnote{Recall that complements are unique in distributive lattices.} 
        Then $aA$ has complement $\c{(aA)} = a\c A + \sum\limits_{b\neq a} b\top$:
       \[
        \begin{array}{r@{\ \implies \ }ll}
            0 = A\cap \c A & 0 = aA \cap a\c A & \text{by \eqref{eq:letter-lbdd-lattice-homo}} \\
            & 0 = (aA \cap a\c A) + \sum\limits_{b\neq a} (aA \cap  b\top) & \text{by \eqref{eq:letters-freely-cogenerate-structure}} \\
            & 0 = aA \cap (a\c A + \sum\limits_{b\neq a} b\top) & \text{by distributivity} \\
            & 0 = aA \cap \c{(aA)} & \text{by definition}
       \end{array}
       \]
Similarly, one can show $\top=A+A^c\implies \top=aA+(aA)^c$.
    \end{example}

However, the issue with the principles thusfar, \cref{eq:dist-lattice,eq:letter-lbdd-lattice-homo,eq:letters-freely-cogenerate-structure,eq:mu-is-least-prefix,eq:nu-is-greatest-postfix}, is that they do not guarantee such duality of $\mu$ and $\nu$.
Let us address this issue now.

\subsection{Incompleteness strikes!}
Not all models of \cref{eq:dist-lattice,eq:letter-lbdd-lattice-homo,eq:letters-freely-cogenerate-structure,eq:mu-is-least-prefix,eq:nu-is-greatest-postfix} interpret $e$ and $\c e$ as complements. 
Indeed it is well known that there are even completely distributive lattices, let alone models of \cref{eq:dist-lattice,eq:letter-lbdd-lattice-homo,eq:letters-freely-cogenerate-structure,eq:mu-is-least-prefix,eq:nu-is-greatest-postfix}, that are not even Heyting algebras, let alone Boolean algebras.
Still, this does not quite yet give unprovability of the complementary laws for closed expressions (which carve out a substructure of a model). 
Indeed in even complete distributive lattices $\mu $ and $\nu$ are at least dual, in the sense that they \emph{preserve} complements. 
Let us develop an appropriate counterexample, exploiting the incompleteness of the lattice structure:

\begin{example}
    [Incompleteness]
    \label{ex:incompleteness-cantor-topology}
 Consider the Cantor topology $\cantor$ on $\Alphabet^\omega$: $A\subseteq \Alphabet^\omega$ is \emph{open} if it is a (possibly infinite) union of sets of form $a_1\cdots a_n \Alphabet^\omega$. 
 $\cantor$ is closed under finite meets and infinite joins, as it is a topology, so it forms a (bounded) join-complete lattice.
 So we have:
 \begin{itemize}
     \item $\cantor$ satisfies \eqref{eq:dist-lattice}, under the usual set-theoretic union and intersection; and,
     \item We can interpret least and greatest fixed points in $\cantor$ by setting, for monotone open operators $F$:
     \begin{itemize}
         \item $\cantor(\mu F) \df \bigcup\limits_{\alpha \in \mathsf{Ord}} F^\alpha(\emptyset)$;\todo{could define in more detail or refer to earlier} and,
         \item $\cantor(\nu F) \df \bigcup\limits_{A\subseteq F(A)} A$.
     \end{itemize}
     where $F^\alpha (X)$ is defined by transfinite induction on $\alpha$ as follows:
     \begin{itemize}
         \item $F^0(X):=X$;
         \item $F^{\alpha+1}:=F(F^{\alpha}(X))$; and,
         \item $F^\lambda(X):=\displaystyle\bigcup_{\beta\in\gamma}F^\beta(X)$ for limit ordinal $\gamma$.
     \end{itemize}
     It is not difficult to see that these interpretations of $\mu F$ and $\nu F$ are always least/greatest pre/post fixed points, respectively, in $\cantor$, as long as $F$ is monotone.
     Thus $\cantor$ furthermore satisfies \cref{eq:mu-is-least-prefix,eq:nu-is-greatest-postfix}.
 \end{itemize}
Now define the homomorphisms $a\in \Alphabet$ in $\cantor$ just as in $\Lang$: $aA \df \{aw : w \in A\}$. Clearly this is an open map and, under this interpretation, $\cantor$ satisfies \cref{eq:letter-lbdd-lattice-homo,eq:letters-freely-cogenerate-structure} as it is a substructure of $\Lang$.

However the denotation of greatest fixed points in $\cantor $ may be smaller than in $\Lang$, as its definition as a union of postfixed points ranges over only open sets, not all languages.
Indeed we have:
\begin{itemize}
    \item $\cantor (\nu X (aX)) = \emptyset$. For this, reasoning in $\cantor$, note that surely $\nu X (aX) \leq \top$ by boundedness, and so $\nu X (aX) \leq a^n\top$ for all $n\in \Nat$, by monotonticity and since $\nu X(aX) $ is a fixed point of $X\mapsto aX$.
    The only nonempty subset of $\Alphabet^\omega$ satisfying this property is $\{a^\omega\}$, but this is not open and so does not belong to $\cantor$.
    On the other hand, evidently $a\emptyset = \emptyset$.
    \item $\cantor{\c{(\nu X (aX))}} \neq \Alphabet^\omega$. Reasoning in $\cantor$, we have that $\c{(\nu X (aX))} = \mu X (aX + \sum\limits_{b\neq a}b\top) $, which (necessarily) has the same denotation in $\cantor $ as in $\Lang$: the set of words with at least one letter $b\neq a$. 
\end{itemize}
Thus $\nu X(aX)$ and $\c{(\nu X (aX))} $ are not complementary in $\cantor$.
Since $\cantor$ is a model of \cref{eq:dist-lattice,eq:letter-lbdd-lattice-homo,eq:letters-freely-cogenerate-structure,eq:mu-is-least-prefix,eq:nu-is-greatest-postfix}, it is immediate that this set of axioms is incomplete for $\Lang$: it does not prove $\top = \nu X(aX) + \c{(\nu X (aX))} $.
\end{example}

The issue for \cref{eq:dist-lattice,eq:letter-lbdd-lattice-homo,eq:letters-freely-cogenerate-structure,eq:mu-is-least-prefix,eq:nu-is-greatest-postfix}, towards completeness for $\Lang$, is that, in the absence of completeness of the lattice, it is not immediately clear that $\mu$ and $\nu$ are dual. 
Duality is derivable for $+ $ and $\cap$ from \cref{eq:dist-lattice}, but the infinitary nature of the fixed points means that it does not follow as a consequence of \cref{eq:dist-lattice,eq:letter-lbdd-lattice-homo,eq:letters-freely-cogenerate-structure,eq:mu-is-least-prefix,eq:nu-is-greatest-postfix}. 
\section{An axiomatisation}
\label{sec:axiomatisation}

In this section, we will develop an axiomatisation $\RLLLang$ for equations over RLL expressions that are valid in $\Lang$. Towards a definition of our ultimate axiomatisation, let us give a final property in $\Lang $: 


\begin{itemize}
    \item  $\mu$ and $\nu$ are dual:
      \begin{equation}
        \label{eq:mu-nu-are-dual}
        \begin{array}{l}
             \forall X,Y ( \top \leq X + Y \implies \top \leq e(X) + f(Y)) \implies \top \leq \mu X e(X) + \nu Y f(Y) \\
             \forall X,Y (X\cap Y \leq 0 \implies e(X) \cap f(Y) \leq 0 ) \implies \mu X e(X) \cap \nu Y f(Y) \leq 0
        \end{array}
    \end{equation}
\end{itemize}


It is not difficult to see that the above principles hold in any completely distributive lattice, not just in $\Lang$, by induction on the closure ordinals of fixed points. 
However, unlike completeness, the principle above is first-order, not second-order.
Note also that the principle above does not state the \emph{existence} of complements, just that $\mu$ and $\nu$ behave well wrt complements in the same way that $+$ and $\cap $ do. For all these reasons it is quite natural to include \eqref{eq:mu-nu-are-dual} natively within any `right linear lattice axiomatisation' for $\Lang$. We are now ready to axiomatise the right-linear lattice theory of $\Lang$.

\begin{definition}
    Write $\RLLLang$ for the theory axiomatised by \cref{eq:dist-lattice,eq:letter-lbdd-lattice-homo,eq:letters-freely-cogenerate-structure,eq:mu-is-least-prefix,eq:nu-is-greatest-postfix,eq:mu-nu-are-dual}.
\end{definition}

Our main result is that this axiomatisation is indeed sound and complete for the RLL theory of $\Lang$:
\begin{theorem}[Soundness and completeness of $\RLLLang$]
\label{thm:RLL-soudness-completeness}
    $\Lang \models e=f$ $\iff$ $\RLLLang \proves e=f$.
\end{theorem}

Let us point out that the soundness direction, $\impliedby$, follows from the commentary introducing each of the axioms \cref{eq:dist-lattice,eq:letter-lbdd-lattice-homo,eq:letters-freely-cogenerate-structure,eq:mu-is-least-prefix,eq:nu-is-greatest-postfix,eq:mu-nu-are-dual}. 
For the completeness direction, $\implies$, we shall reduce to the completeness result for the fixed point logic $\muLTL$.~\Cref{sec:completeness} is dedicated to proving this formally. Before that, let us establish some properties of $\RLLLang$.

\begin{proposition}[Functoriality]
\label{cor:functoriality}
$\RLLLang \proves f \leq g \implies e(f) \leq e(g)$.
\end{proposition}

\begin{proof}
We will prove a stronger statement \viz for all $i$, $\Vec{f_i}\le \Vec{g_i}$ $\implies$ $e(\Vec{f_i})\le e(\Vec{g_i})$. We will prove by induction on $e(\Vec{X})$.
\begin{itemize}
    \item When $e=X$, what is to be proved is literally the hypothesis.
    \item Suppose $e=ae_0(\Vec{X})$. By induction hypothesis, $e_0(\Vec{f_i})\le e_0(\Vec{g_i})$ or, $e_0(\Vec{f_i})+e_0(\Vec{g_i})=e_0(\Vec{g_i})$. Therefore, by~\Cref{eq:letter-lbdd-lattice-homo}, $ae_0(\Vec{f_i})+ae_0(\Vec{g_i})=ae_0(\Vec{g_i})$, or $ae_0(\Vec{f_i})\le ae_0(\Vec{g_i})$.
    \item Suppose $e=e_0(\Vec{X})+e_1(\Vec{X})$. By induction hypothesis, $e_0(\Vec{f_i})\le e_0(\Vec{g_i})$ and $e_1(\Vec{f_i})\le e_1(\Vec{g_i})$. Similarly, as before, we can reason under inequalities by converting them into equalities. So, we have $e_0(\Vec{f_i})+e_1(\Vec{f_i})\le e_0(\Vec{g_i})+e_1(\Vec{g_i})$. Similarly for the case when $e=e_0\cap e_1$.
    \item Suppose $e=\mu X e_0(X,\Vec X)$. By induction hypothesis, $e_0(\mu Xe_0(X,\Vec{g_i}),\Vec{f_i})\le e_0(\mu Xe_0(X,\Vec{g_i}),\Vec{g_i})$. By prefix, $e_0(\mu Xe_0(X,\Vec{g_i}),\Vec{f_i})\le \mu Xe_0(X,\Vec{g_i})$. By induction, $\mu Xe_0(X,\Vec{f_i})\le \mu Xe_0(X,\Vec{g_i})$. When $e=\nu Xe_0(X,\Vec X)$ it is symmetric. \qedhere
\end{itemize}
\end{proof}

As an immediate corollary of functoriality, we have:

\begin{example}[Fixed points are fixed points]
\label{ex:postfix-mu}
By a standard argument mimicking the proof of the Knaster-Tarski theorem, $\RLLLang\proves\mu X e(X) \leq e (\mu X e(X))$ and dually, $\RLLLang\proves e(\nu X e(X)) \leq \nu X e(X)$. We will show the first one. By Induction it suffices to show that $e(\mu Xe(X)) $ is a prefixed point, i.e.\ $e(e(\mu Xe(X))) \leq e(\mu X e(X))$. Now, by the functors of \cref{cor:functoriality} above it suffices to show $e(\mu Xe(X))\leq \mu X e(X)$, which is just the Prefix axiom.
\end{example}

We will now show the provable correctness of the syntactic notion of complementation we introduced at the beginning of this section:

\begin{proposition}
    [Complementation]
    \label{prop:compl-is-compl}
    $\RLLLang$ proves the following, for all closed $e$:
    \begin{equation}
        \label{eq:c-is-complement}
        \begin{array}{r@{\ \leq \ }l}
             \top & e + \c e \\
             e \cap \c e & 0
        \end{array}
    \end{equation}
\end{proposition}

The result follows immediately from the following lemma more generally establishing `complement functoriality', by setting $\vec X$ and $\vec Y$ to be empty in:
\begin{lemma}
    \label{lem:compl-functors}
    $\RLLLang $ proves 
    \begin{equation}
        \label{eq:compl-functors}
        \begin{array}{l}
        \forall \vec X,\vec Y (\bigwedge_i \top \leq X_i + Y_i \implies \top \leq e(\vec X) + \c e (\vec Y)) \\
        \forall \vec X,\vec Y (\bigwedge_i X_i \cap Y_i \leq 0 \implies e(\vec X) \cap \c e (\vec Y) \leq 0)  
        \end{array}
    \end{equation}
\end{lemma}

\begin{proof}
[Proof sketch]
    By induction on $e(\cdot)$.
    When the outermost connective of $e$ is a $+$ or $\cap$ we appeal to the induction hypothesis by duality of $+$ and $\cap$ more generally in bounded distributive lattices.
    The case when $e$ has form $af$ is handled similarly to \cref{ex:compl-of-action-in-lang}, only with the presence of free variables.
    It remains to check the fixed point cases.

Suppose $e(\vec X)$ has form $\mu X f(X,\vec X)$. 
Reasoning in $\RLLLang$, suppose $\top \leq X_i + Y_i$ and $X_i \cap Y_i \leq 0$ for all $i$. 
We have:
\[
\begin{array}{rll}
     & \forall X,Y (\top \leq X+Y \implies \top \leq f(X,\vec X) + \c f(Y ,\vec Y)) & \text{by IH} \\
     \therefore & \top \leq \mu X f(X,\vec X) + \nu X \c f(X,\vec Y) & \text{by \eqref{eq:mu-nu-are-dual}}
\end{array}
\]
\[
\begin{array}{rll}
    & \forall X,Y (X\cap Y  \leq 0  \implies   f(X,\vec X) \cap \c f (Y, \vec Y) \leq 0 ) & \text{by IH} \\
    \therefore & \mu X f(X,\vec X) \cap \nu X \c f (X,\vec X) \leq 0 & \text{by \eqref{eq:mu-nu-are-dual}}
\end{array}
\]
The argument for the case when $e(\vec X)$ has form $\nu X f(X,\vec X)$ is symmetric.
\end{proof}

We end this section with some examples of models of $\RLLLang$.

In \cref{sec:complement} we defined a complement expression $\c e$  of each RLL expression $e$, and \cref{prop:compl-is-compl} showed that $e$ and $\c e$ are provable complementary in $\RLLLang$. This means that any model of $\RLLLang$ has a substructure, namely the denotations of RLL expressions, that forms a Boolean algebra. The same holds for Kleene Algebras, as each regular expression can also be associated with one computing its complement, with respect to the regular language model. Just like $\KA$, this does not mean that all models of $\RLLLang$ are Boolean algebras themselves.

\begin{example}[$\RLLLang$ model without general complements]
Fix the alphabet $\{0,1\}$. Consider the substructure $\mathcal K$ of $\Lang$ 
that is the smallest $\bigcup$-complete lattice containing every $\omega$-regular language and $Q:= (0,1)\cap \mathbb Q$. First, note that indeed $\mathcal K \models \RLLLang$:
\begin{itemize}
    \item \cref{eq:dist-lattice,eq:letter-lbdd-lattice-homo,eq:letters-freely-cogenerate-structure} hold as $\mathcal K \leq \mathcal L$.
\item For \eqref{eq:mu-is-least-prefix}, we define $(\mu X e(X))^\mathcal K := \bigcup\limits_{\alpha \in \mathsf{Ord}} e^\alpha(\emptyset)$.
This is well defined and coincides with $\lang{\mu X e(X)}$ by $\bigcup$-completeness and the approximant definition of the latter.
\item For \eqref{eq:nu-is-greatest-postfix}, we define $(\nu X e(X))^\mathcal K := \bigcup \{A \subseteq e(A)\} $.
Since, in particular, $\lang {\nu X e(X)}$ is a postfixed point and an $\omega$-regular language, it must coincide with $(\nu X e(X))^\mathcal K$.
\end{itemize}

However it is not hard to see that $Q$ does not have a complement in $\mathcal K$, i.e.\ that $(0,1)\setminus \mathbb Q$ does not belong to $\mathcal K$.
For this note that, as powerset lattices are completely distributive (and therefore so are their (semi)complete sublattices), we can write any element $A$ of $\mathcal K$ as an infinite union of finite intersections of $\omega$-regular languages and $Q$, i.e.\ of the form $\bigcup\limits_{i\in I} A_{i1} \cap \cdots \cap A_{in_i}$, where each $A_{ij}$ is $\omega$-regular or $Q$.
Now, if $A\neq \emptyset$, then also some $A_i := A_{i1}\cap \cdots \cap A_{in_i} \neq \emptyset$ as well. 
However, since $\omega$-regular languages are closed under intersection, $A_i$ must contain the rational part of some nonempty $\omega$-regular language.
Since any non-empty $\omega$-regular language must contain some ultimately periodic word, this means that $A\cap \mathbb Q \supseteq A_i \cap \mathbb Q \neq \emptyset$, and so $A $ cannot be a complement of $Q$ in $\mathcal K$.    
\end{example}

\begin{example}[Minmax as a model of $\RLLLang$]
\label{ex:minmax-as-rll}

Note that $[0,1]$ with $0:=0$, $\top:=1$, $+:=\max$, and $\cap:=\min$ is a bounded distributive lattice. Let $\Alphabet=\{\id\}$ and define $\id\cdot:x\mapsto x$. It is easy to check that \Cref{eq:letter-lbdd-lattice-homo,eq:letters-freely-cogenerate-structure} are satisfied. Define
\[\mu Xe:= \inf\{x\mid e(x)\le x\}\qquad \nu Xe :=\sup\{x\mid x\le e(x)\}\]
Since $[0,1]$ is compact, $\mu e$ and $\nu e$ exist for any $e$. To prove~\Cref{eq:mu-is-least-prefix,eq:nu-is-greatest-postfix,eq:mu-nu-are-dual}, first note that any function $e:[0,1]^n\to[0,1]$ composed of $\max$, $\min$, and $\id$ is non-decreasing. Let $f,g:[0,1]\to[0,1]$ be non-decreasing functions. Define   $\alpha:=\inf\{x\mid f(x)\le x\}$ and $\beta:=\sup\{x\mid x\le g(x)\}$.

\subparagraph{Prefix.} Suppose $\alpha<f(\alpha)$. Then, there exists $\alpha \le y < f(\alpha)$ such that $f(y)\le y$. Since $f$ is non-decreasing, $f(\alpha)\le f(y)\le y < f(\alpha)$. Contradiction! 

\textbf{Induction.} Need to show that if $f(x)\le x$, then $\alpha\le x$ -- this holds by definition of $\inf$.  

\textbf{Duality.} Suppose $f,g$ are such that whenever $\max(x,y)\ge 1$, we have $\max(f(x),g(y))\ge 1$. We need to show that $\max(f(\alpha),g(\beta))\ge 1$. Since $\max(x,1)\ge 1$ for all $x$, $\max(f(x),g(1))\ge 1$. So, either $f(x)=1$ for all $x$ or $g(1)=1$. In the first case, $f(\alpha)=1$ and hence we are done. In the second case, we have $1\in\{x\mid x\le g(x)\}$. So, $\beta=1$. Therefore, $g(\beta)=g(1)=1$ and we are done. 
%
%
Postfix, Coinduction, and the other case of Duality are symmetric arguments. Note that $[0,1]$ can be replaced by any compact subset of $\mathbb{R}$.

\end{example}

Note that \Cref{eq:letters-freely-cogenerate-structure} is the only axiom that is bespoke to the $\Lang$ interpretation. In fact, we can easily modify~\Cref{ex:minmax-as-rll} to be a model of~\Cref{eq:dist-lattice,eq:letter-lbdd-lattice-homo,eq:mu-is-least-prefix,eq:nu-is-greatest-postfix,eq:mu-nu-are-dual}. Let us call $\RLL$ the theory axiomatised by these equations.

\begin{example}[Minmax as a model of $\RLL$]

Let $\Alphabet=\{a_1,\cdots,a_n\}$ where $a_i\in(0,1)$ for all $i$. Let $a_i\cdot:x\mapsto a_i x$. As before we work with the bounded distributive lattice $[0,1]$ with $0:=0$, $\top:=1$, $+:=\max$, and $\cap:=\min$ ans we have~\Cref{eq:dist-lattice,eq:letter-lbdd-lattice-homo}. Any function $e:[0,1]^n\to[0,1]$ composed of $\max$, $\min$, and $a_i\cdot$ is non-decreasing. Therefore,~\Cref{eq:mu-is-least-prefix,eq:nu-is-greatest-postfix,eq:mu-nu-are-dual} is satisfied. 

However, \Cref{eq:letters-freely-cogenerate-structure} does not hold. Let $e\ne 0$ and $f\ne 0$. Then $\min(a_ie,a_jf)\ne 0$. Similarly, $\max(\{a_i\}_{i=1}^n)\ne 1$. Therefore, this is a model of $\RLL$ and \emph{not} of $\RLLLang$.      
\end{example}

\section{Completeness via \texorpdfstring{$\muLTL$}{}}
\label{sec:completeness}

In this section, we will prove the completeness of $\RLLLang$. Our completeness proof relies on the completeness of an axiomatisation of the linear-time $\mu$-calculus called $\muLTL$. 
We show several syntactic and semantic simulations between $\RLLLang$ and $\muLTL$. For the sake of brevity, we only give the directions necessary to recover completeness of $\RLLLang$ wrt.\ $\Lang$.

\subsection{A (very quick) recap of \texorpdfstring{$\muLTL$}{}}

Linear temporal logic (LTL) is a modal logic with modalities referring to time. In LTL, one can encode formulas about the future of \emph{paths}. In particular, we have formulas of the form $\nxt \phi$ and $\phi \mathbf U \psi$ that are (informally) interpreted as `at the next timestamp $\phi$ holds' and `$\phi$ holds until $\psi$ holds.' Naturally, they are interpreted over linear Kripke structures (\ie the accessibility relation is successor on $\Nat$). Note that $\phi \mathbf U \psi$ can be construed as a fixed point operator $\nu X (\psi \vee (\phi\wedge\nxt X))$. $\muLTL$ is the generalisation of LTL with arbitrary fixed points.   

$\muLTL$ formulas, written $\phi,\psi,\dots$, are generated by:
\[
\phi,\psi,\dots
\quad ::= \quad
\bot \ \mid \ \top \ \mid \ P \ \mid \ \bar P \  \mid \ X \  \mid \  \phi \lor \psi \  \mid \  \phi \land \psi \ \mid \ 
\nxt \phi
\  \mid \  \mu X \phi \  \mid \ \nu X \phi
\]

Readers familiar with the modal $\mu$-calculus might think of $\muLTL$ as a fragment of $\mu$-calculus with a self-dual modality. We will define the semantics over the canonical linear Kripke structures \viz $\omega$. We shall assume that the propositional letters $P,Q,\dots$ are from some finite set $\props$. It is pertinent now to fix an alphabet $\Alphabet = \pow \props$.

\begin{definition}[Semantics of $\muLTL$]
Let us temporarily expand the syntax of formulas by a constant symbol $\alpha$ for each subset $\alpha\subseteq \omega$. For $\omega$-words $\sigma \in \Alphabet^\omega $ ( i.e.\  $\sigma \in \pow \props^\omega$) and formulas $\phi$, we define $\interp \sigma \phi \subseteq \omega $ by:
\begin{align*}
&\interp \sigma \bot \df \emptyset    && \interp \sigma \top \df \omega\\
&\interp \sigma P \df \{ n\in \omega : P \in \sigma_n\} && \interp \sigma {\bar P} \df \{ n\in \omega : P \notin \sigma_n\} \\
&\interp \sigma \alpha \df \alpha\\
&\interp \sigma {(\phi \land \psi)} \df \interp \sigma \phi \cap \interp \sigma \psi &&\interp \sigma {(\phi \land \psi)} \df \interp \sigma \phi \cap \interp \sigma \psi\\ 
&\interp {\sigma} {(\nxt \phi)} \df \{ n \in \omega : n+1 \in \interp \sigma \phi\}\\
&\interp \sigma {(\mu X \phi(X))} \df \bigcap \{ A \supseteq \interp \sigma {\phi(A)} \} && \interp \sigma {(\nu X \phi(X))} \df \bigcup \{ A \subseteq \interp \sigma {\phi(A)} \}
\end{align*}
%
Write $\sigma \models \phi$ if $0 \in \interp \sigma \phi$. We say $\phi$ is \defname{valid}, written $ \models \phi$, if for all $\sigma \in \Alphabet^\omega$ we have $\sigma \models \phi$.
\end{definition}


$\muLTL$ enjoys a sound and complete axiomatisation~\cite{Kaivola95,Doumane17}.
To recast this axiomatisation in the current logical basis, let us point out that we can extend negation to all $\muLTL$ formulas by defining $\bar \phi$ exploiting De Morgan duality of $\bot,\top$ and $\lor,\land$ and $\mu,\nu$, and finally self-duality of $\nxt$: $\overline{\nxt \phi} \df \nxt \bar \phi$.
Therefore, we may freely use other propositional connectives such as $\lnot,\limp,\liff$ as suitable macros. The following axiomatisation is equivalent to that of~\cite{Kaivola95}, only adapted to our negation normal syntax.

\begin{figure}[t]
    \centering
\begin{tabular}{l|cc}
\multirow{3}{*}{\textbf{Axioms}} & \multicolumn{2}{l}{All propositional tautologies}\\[5pt]  
& $\nxt (\phi \lor \psi) \liff \nxt \phi \lor \nxt \psi$ & $\nxt (\phi \land \psi)\liff \nxt \phi \land \nxt \psi$ \\[5pt]
& $\phi(\mu X \phi(X)) \limp \mu X\phi(X)$ & $\nu X\phi(X) \limp \phi(\nu X\phi(X))$\\[5pt]\hline\\
\multirow{2}{*}{\textbf{Rules}} & $\vliinf{MP}{}{\psi}{\phi}{\phi \limp \psi}$ & $\vlinf{\nxt}{}{\nxt \phi}{\phi}$\\[10pt]
& $\vlinf{\mu}{}{\mu X \phi(X) \limp \psi}{\phi(\psi) \limp \psi}$ & $\vlinf{\nu}{}{\psi \limp \nu X \phi(X)}{\psi \limp \phi(\psi)}$
\end{tabular}

    \caption{A Hilbert-style axiomatisation of $\muLTL$}
    \label{fig:muLTL-axiomatisation}
\end{figure}

\begin{definition}[Hilbert-style axiomatisation of $\muLTL$] $\muLTL$\footnote{By abuse of notation, we refer to both the language and the axiomatisation as $\muLTL$.} is defined as the set of instances of the axioms closed under the inference rules in~\cref{fig:muLTL-axiomatisation}. 
\end{definition}

\begin{example}
Recall that $\phi\U\phi:=\nu X (\psi \vee (\phi\wedge\nxt X))$. We will prove the LTL tautology $\nxt(\phi\U\psi)\limp\nxt\phi\U\nxt\psi$. First note that the following modal rule is derivable
\[\vlinf{(\star)}{}{\nxt\phi\limp\nxt\psi}{\phi\limp\psi}\]
Thus we have,
\begin{align*}
\nxt(\phi\U\psi)&\limp\nxt(\psi\vee(\phi\wedge\nxt(\phi\U\psi))) &&\text{by }(\star)\text{ and }\nu\text{-unfolding}\\
                &\limp\nxt\psi\vee\nxt(\phi\wedge\nxt(\phi\U\psi)) &&\text{by normality of}\nxt\text{ over }\vee\\
                &\limp\nxt\psi\vee(\nxt\phi\wedge\nxt\nxt(\phi\U\psi)) &&\text{by normality of}\nxt\text{ over }\wedge
\end{align*}
Applying the $\nu$ rule, we are done.
\end{example}

\begin{theorem}[\cite{Kaivola95}]
\label{kaivola}
$\muLTL$ is sound and complete \ie  $\muLTL \proves \phi$ $\iff$  $\models \phi$.
\end{theorem}

\subsection{Interpreting \texorpdfstring{$\RLLLang$}{} in \texorpdfstring{$\muLTL$}{} and vice versa}
Our aim is to reduce the completeness of $\RLLLang$ to that of $\muLTL$.
For this reason we need to embed $\RLLLang$  into $\muLTL$.

\begin{definition}
    For (possibly open) RLL expressions $e$ we define a $\muLTL$ formula $\ef e $ by induction on the structure of $e$ as follows:
    \begin{itemize}
        \item $\ef X \df X$
        \item $\ef {(ae)} \df \bigwedge\limits_{P\in a} P \land \bigwedge\limits_{P\notin a} \bar P\land \nxt \ef e$
        \item $\ef{e+f} \df \ef e \lor \ef f$
        \item $\ef {e\cap f} \df \ef e \land \ef f $
        \item $\ef {(\mu X e)} \df \mu X \ef e$
        \item $\ef {(\nu X e )} \df \nu X \ef e $
    \end{itemize}
\end{definition}

We need to show that the translation above is faithful wrt.\ the two semantics we have presented, for RLL expressions and for $\muLTL$ formulas. Writing $\lang \phi \df \{\sigma \models \phi\}$ for closed $\muLTL$ formulas $\phi$, we have:

\begin{proposition}[Semantic adequacy]
\label{adequacy}
$\lang e \subseteq \lang{\ef e}$, for closed expressions $e$.
\end{proposition}

To prove this, we must first address the fact that our two semantics interpret syntax as different types of sets, and duly have different types of constant symbols. 
To this end, let us temporarily introduce into the language of $\muLTL$ a constant symbol $A$ for each language $A\subseteq \Alphabet^\omega$.
We extend the definition of $\ef -$ by the clause $\ef A \df A$ and duly extend the definition of $\interp \sigma -$ by the clause $\interp \sigma A \df \{n\in \omega : \sigma^n \in A\}$
where $\sigma^n$ is the $n$\textsuperscript{th} tail of $\sigma$, i.e.\ we set $\sigma^0 \df \sigma$, and $\sigma^{n+1}$ to be the tail of $\sigma^n$.
%
%
Now we can establish a sort of substitution lemma that relates our two semantics:

\begin{lemma}[Mixed substitution]
\label{mixed-subst-lem}
    $\interp \sigma {\phi (\lang{ \chi})} \subseteq \interp \sigma {\phi( \chi)}$.
\end{lemma}

\begin{proof}
By Induction on the size of $\phi(X)$, i.e.\ its number of symbols. 
\begin{itemize}
        \item If $\phi(X)$ is a variable $X$ then:
        \[
        \begin{array}{r@{\ \implies \ }ll}
             n \in \interp \sigma{ \lang {\chi}} & \sigma^n \in \lang{\chi} & \text{by definition of $\interp \sigma -$} \\
                & \sigma^n \models \chi & \text{by definition of $\lang - $} \\
                & 0 \in \interp{\sigma_n} {\chi} & \text{by definition of $\models$} \\
                & n \in \interp \sigma \chi & \text{by properties of $\interp \sigma -$}
        \end{array}
        \]
        \item The cases when $\phi(X)$ is an atomic formula (that is not $X$), a disjunction or conjunction are routine.
        \item If $\phi(X)$ is $\nxt \psi(X)$ then:
        \[
        \begin{array}{rr@{\ \implies \ }ll}
             & \forall n \  [\  n \in \interp \sigma {\psi (\lang \chi)} & n \in \interp \sigma {\psi(\chi)} \ ] & \text{by Induction hypothesis} \\
             \therefore & n+1 \in \interp \sigma {\psi (\lang \chi)} & n+1 \in \interp \sigma {\psi(\chi)} & \text{by $\forall$ instantiation} \\
             \therefore & n \in \interp \sigma {(\nxt \psi(\lang \chi))} & n \in \interp \sigma {(\nxt \psi(\chi))} & \text{by definition of $\interp \sigma -$}
        \end{array}
        \]
        \item If $\phi( X)$ is $\mu Y \psi(X, Y)$ then: 
        \[
        \begin{array}{rr@{\ \subseteq \ }ll}
             & \interp \sigma {\psi(\lang \chi, \interp \sigma {(\mu Y \psi (\chi,Y))})} &  \interp \sigma {\psi (\chi,\interp \sigma {(\mu Y \psi (\chi,Y))}) } & \text{by Induction hypothesis} \\
            &   & \interp \sigma {\psi (\chi, {\mu Y \psi (\chi,Y)}) } & \text{by substitution property of $\interp \sigma -$} \\
            & & \interp \sigma {(\mu Y \psi (\chi,Y))} & \text{since $\interp \sigma \mu$ is a prefixed point} \\
            \therefore & \interp \sigma {(\mu Y \psi (\lang \chi, Y))} & \interp \sigma {(\mu Y \psi (\chi, Y)} & \text{by $\interp\sigma \mu$-induction}
        \end{array}
        \]
        \item If $\phi(X)$ is $\nu Y \psi (X,Y)$ then:
        \[
        \begin{array}{rr@{\ \subseteq \ }ll}
             & \interp \sigma {\psi (\lang \chi, \interp \sigma {(\nu Y \psi (\lang \chi, Y))})} & \interp \sigma {\psi (\chi, \interp \sigma {(\nu Y \psi (\lang \chi, Y))})} & \text{by Induction hypothesis} \\
             & \interp \sigma {\psi (\lang \chi,  {\nu Y \psi (\lang \chi, Y)})} &  & \text{by substitution property of $\interp \sigma -$} \\
             & \interp \sigma {(\nu Y \psi (\lang \chi, Y))} & & \text{since $\interp \sigma \nu$ is a postfixed point} \\
        \therefore & \interp \sigma {(\nu Y \psi (\lang \chi, Y))} & \interp \sigma {(\nu Y \psi (\chi , Y)} & \text{by $\interp \sigma \nu$-coinduction}

        \end{array}
        \]
    \end{itemize}
\end{proof}

Now, semantic adequacy is readily proved:

\begin{proof}[Proof of~\Cref{adequacy}]
We proceed by induction on the size of $e$.
\begin{itemize}
    \item If $e$ is a constant symbol $A \subseteq \Alphabet^\omega$, then:
    \[
    \begin{array}{r@{\ \implies \ }ll}
         \sigma \in A & \sigma \in \ef A &\text{by definition of $\ef -$}  \\
            & \sigma^0 \in \ef A & \text{by definition of $-^n$} \\
            & 0 \in \interp \sigma {{\ef A}} & \text{by definition of $\interp \sigma -$} \\
            & \sigma \in \lang {\ef A} & \text{by definition of $\lang -$} 
    \end{array}
    \]
    \item If $e$ is $af$ then:
        \[
        \begin{array}{r@{\ \implies \ }ll}
             a\sigma \in \lang {af} & \sigma \in \lang{f} & \text{by definition of $\lang \cdot$} \\
                & \sigma \models \ef {f } & \text{by Induction hypothesis} \\
                & a\sigma  \models \nxt \ef {f } & \text{by definition of $\models$} \\
                & a\sigma  \models \bigwedge\limits_{P\in a}P \land \bigwedge\limits_{P\notin a} \bar P  \land \nxt \ef {f} & \text{by definition of $\models$} \\
                & a\sigma \models \ef{(af)} & \text{by definition of $\ef - $ and $\models$}
        \end{array}
        \]
        \item The cases when $e$ is a $+$ or $\cap$ expression are routine.
        \item If $e$ is $\mu Xf(X)$ then:
        \[
        \begin{array}{rr@{\ \subseteq \ }ll}
             & \lang {f(\lang {\ef{(\mu X f(X))}})} & \lang {\ef {f(\lang{\ef{(\mu X f(X))}} } } & \text{by Induction hypothesis} \\
            \therefore & \lang {f(\lang {{\mu X \ef f(X)}})} &   \lang {\ef f(\lang{{\mu X \ef f(X)} } } & \text{by definition of $\ef -$} \\
            & &  \lang{\ef f (\mu X \ef f(X))} & \text{by \cref{mixed-subst-lem}} \\
            & & \lang{\mu X \ef f(X)} & \text{since $\lang\mu$ is a prefixed point} \\
        \therefore & \lang{\mu X f(X)} & \lang{\mu X \ef f(X)} & \text{by $\lang \mu$-induction}
        \end{array}
        \]
        \item If $e$ is $\nu Xf(X)$ then:
        \[
        \begin{array}{rr@{\ \subseteq \ }ll}
             & \lang {f(\lang {\nu Xf(X)})} & \lang {\ef {f (\lang {\nu Xf(X)})}} & \text{by Induction hypothesis} \\
             & & \lang {\ef f (\lang {\nu Xf(X)})} & \text{by definition of $\ef - $}  \\
            & \lang {\nu X f(X)} & & \text{since $\lang \nu$ is a postfixed point} \\
         \therefore & \interp \sigma {\lang {\nu X f(X)}} & \interp \sigma {\lang {\ef f (\lang {\nu X f(X)})}}  & \text{by monotonicity property of $\interp \sigma -$} \\
            & &  \interp \sigma {\ef f (\lang {\nu X f(X)})} & \text{by \cref{mixed-subst-lem}} \\
            & & \interp \sigma {\ef f (\interp \sigma {\lang {\nu X f(X)}})} & \text{by substitution property of $\interp \sigma -$}\\
        \therefore & \interp \sigma {\lang {\nu X f(X)}}  & \interp \sigma {(\nu X \ef f (X))} & \text{by $\interp \sigma \nu$-coinduction} \\
            & & \interp \sigma {{\ef {(\nu X f(X))}}} & \text{by definition of $\ef -$}
        \end{array}
        \]
        So in particular,  $\sigma \in \lang {\nu Xf(X)}\implies 0\in \interp \sigma {\lang{\nu Xf(X)}} \implies 0 \in \interp \sigma {{\ef {(\nu X f(X))}}} \implies \sigma \models  {{\ef {(\nu X f(X))}}} \implies \sigma \in \lang {\nu X f(X)}$. \qedhere
\end{itemize}
\end{proof}

In order to leverage the completeness of $\muLTL$ within $\RLLLang$, we need to simulate its reasoning, for which we must embed $\muLTL$ back into $\RLLLang$.

\begin{definition}
    For $\muLTL$ formulas $\phi$ we define an RLL expression $\fe \phi$ by induction on the structure of $\phi$ as follows:
\begin{align*}
    &\fe \bot \df 0 && \fe \top \df \top\\
    &\fe P \df \sum\limits_{a \ni P } a\top && \fe {\bar P} \df \sum\limits_{a \not \ni P} a\top\\
    &\fe X \df X\\
    &\fe {(\phi\lor \psi)} \df \fe \phi + \fe \psi && \fe {(\phi \land \psi)} \df \fe \phi \cap \fe \psi\\
    &\fe {(\nxt \phi)} \df \sum\limits_{a \in \Alphabet} a\fe \phi\\
    &\fe {(\mu X e)} \df \mu X \fe e && \fe {(\nu X e)} \df \nu X \fe e
\end{align*}
%
\end{definition}


We can again establish the adequacy of this interpretation, though this time we need a syntactic result rather than a semantic one:
\begin{theorem}
    [Syntactic adequacy]
    \label{rll-sim-multl}
    $\muLTL \proves \phi \implies \RLLLang \proves \fe \phi = \top$. 
\end{theorem}

\begin{proof}
    By induction on $\muLTL$ proofs.
    
    \begin{itemize}
        \item All the propositional axioms are handled by the fact that RLL expressions $\RLLLang$-provably form a Boolean Algebra (\cf~\Cref{sec:complement}), and since $\bullet$ is defined directly as a homomorphism $(\bot,\top,\lor,\land) \to (0,\top, +,\cap)$. 
        We also need duality of $\fe P$ and $\fe{\bar P}$ in $\RLLLang$:
        \[
        \begin{array}{r@{\ = \ }l}
             \fe P + \fe {\bar P} & \sum\limits_{a \ni P}a\top + \sum\limits_{a\not\ni P}a\top \\
                & \sum\limits_{a\in \Alphabet}a\top \\
                & \top
        \end{array}
        \qquad
        \begin{array}{r@{\ = \ }l}
             \fe P \cap \fe {\bar P} & \sum\limits_{a\ni P} a\top \cap \sum\limits_{b\ni P}b\top \\
                & \sum\limits_{a\ni P}\sum\limits_{b\not\ni P} a\top \cap b\top\\
                & 0
        \end{array}
        \]
        \item For normality of $\nxt $ wrt $\lor$, it suffices by Boolean reasoning in $\RLLLang$ to derive:
        \[
        \begin{array}{r@{\ = \ }ll}
          \fe {(\nxt(\phi \lor \psi))} & \sum\limits_{a \in \Alphabet} a (\fe \phi + \fe \psi) & \text{by definition of $\fe -$} \\
          & \sum\limits_{a \in \Alphabet} \left( a \fe \phi + a\fe \psi \right)& \text{$\because$ $a$ is a $+$-homomorphism} \\
          & \sum\limits_{a \in \Alphabet} a \fe \phi + \sum\limits_{a \in \Alphabet} a \fe \psi & \text{by commutativity and associativity of $+$} \\
          & \fe {(\nxt \phi \lor \nxt \psi)} & \text{by definition of $\fe -$}
        \end{array}
        \]

        \item For normality of $\nxt$ wrt $\land$, it again suffices by Boolean reasoning in $\RLL$ to derive $\fe{(\nxt (\phi \land \psi))} = \fe{(\nxt \phi \land \nxt \psi)}$:
        \[
        \begin{array}{r@{\ = \ }ll}
           \fe{(\nxt (\phi \land \psi))} & \sum\limits_{a\in \Alphabet} a (\fe \phi \cap \fe \psi) & \text{by definition of $\fe -$} \\
           & \sum\limits_{a\in \Alphabet} \left(a \fe \phi \cap a\fe \psi\right) & \text{$\because$ $a$ is a $\cap$-homomorphism} \\
           & \sum\limits_{a \in \Alphabet} \sum\limits_{b \in \Alphabet} \left( a\fe \phi \cap b \fe \psi\right) & \text{$\because$ $ae \cap bf = 0$ whenever $a\neq b$} \\
           & \sum\limits_{a \in \Alphabet}a\fe \phi \, \cap \, \sum\limits_{b\in \Alphabet}b\fe \psi & \text{by distributivity} \\
           & \fe{(\nxt \phi \land \nxt \psi)}
        \end{array}
        \]
        \item The simulation of axioms for $\mu$ and $\nu$ are immediate, by functoriality, as $\fe -$ commutes with $\mu$ and $\nu$.
        \item Obtaining the rules is mostly straightforward. Modus ponens reduces to transitivity of $\leq$, under Boolean reasoning. Necessitation is simulated by $\top = \sum\limits_{a\in \Alphabet} a\top$. Simulating (co)induction rules are immediate as $\fe -$ commutes with $\mu$ and $\nu$. \qedhere
    \end{itemize}
\end{proof}

\subsection{Compatibility of interpretations and completeness}
To complete our reduction of $\RLLLang$ completeness to $\muLTL$ completeness, as well as simulating $\muLTL$ reasoning, we need compatibility of the two translations.

\begin{proposition}
    [Compatibility]
    \label{ef-fe-compatible}
    $\RLLLang \proves \fe{{\ef e}}=e$
\end{proposition}
\begin{proof}
    By induction on the structure of $e$.
    Almost all cases are immediate, as $\fe{{\ef -}}$ commutes with $X,+,\cap,\mu,\nu$. 
    For the remaining homomorphism case, we reason in $\RLL$:
    \[
        \begin{array}{r@{\ = \ }ll}
             \fe{{\ef {(ae)}}} & \fe {\left( \bigwedge\limits_{P\in a} P \land \bigwedge\limits_{P\notin a} \bar P \land \nxt \ef e \right)} & \text{by definition of $\ef -$} \\
             & \bigcap\limits_{P\in a} \sum\limits_{b\ni P} b\top \cap \bigcap\limits_{P\notin a}\sum\limits_{b\not\ni P} b\top \cap \sum\limits_{c\in \Alphabet} c\fe{{\ef e}} & \text{by definition of $\fe -$} \\
             & a\top \cap \sum\limits_{c\in \Alphabet} c\fe{{\ef e}} & \text{by set theoretic reasoning} \\
             & \sum\limits_{c\in \Alphabet} (a\top \cap c\fe{{\ef e}} ) & \text{by distributivity} \\
             & a\top \cap a\fe{{\ef e}} & \text{since $ae\cap bf = 0$ when $a\neq b$} \\
             & a(\top \cap \fe{{\ef e}}) & \text{as $a$ is a $\cap$-homomorphism} \\
             & a \fe{{\ef e}} & \text{as $\top$ is a $\cap$-unit} \\
             & ae & \text{by induction hypothesis}
        \end{array}
        \]
        To explain a little further the third line above, note that any $b\neq a$ is distinguished from $a$ by either some $P\in a\setminus b$ or some $P\in b\setminus a$.
\end{proof}

We can finally assemble our main completeness result which immediately gives us~\cref{thm:RLL-soudness-completeness}.

\begin{theorem}[Completeness of $\RLLLang$]
$\lang e = \lang f \implies \RLLLang \proves e=f$.
\end{theorem}
\begin{proof}
    By Boolean reasoning it suffices to show that $\lang e = \Alphabet^\omega \implies \RLL \proves  e = \top$:
    \[
    \begin{array}{r@{\ \implies\ }ll}
         \lang e = \Alphabet^\omega & \sigma \models \ef e & \text{by \cref{adequacy}} \\
            & \muLTL \proves \ef e & \text{by \cref{kaivola}} \\ 
            & \RLL \proves \fe {{\ef e }} = \top & \text{by \cref{rll-sim-multl}} \\
            & \RLL \proves e = \top & \text{by \cref{ef-fe-compatible}}
    \end{array}
    \]
\end{proof}
\section{Concluding remarks and future work}
\label{sec:conclusion}

In this work, we introduced RLL expressions, a notation for APAs and gave a sound and complete axiomatisation for their equational theory. We make some observations about our choice of axioms and compare with existing literature.

\subsection{Alternative axiomatisation(s)}

Our axiomatisation $\RLLLang$ for $\Lang$ is first-order, avoiding second-order axioms such as completeness of lattices. Still, stating the duality of $\mu$ and $\nu$, \cref{eq:mu-nu-are-dual}, requires quantifiers.

Let us point out that the completeness argument for $\RLLLang$ only used the principles \eqref{eq:c-is-complement}, an equational consequence of \eqref{eq:mu-nu-are-dual} under \cref{eq:dist-lattice,eq:letter-lbdd-lattice-homo,eq:letters-freely-cogenerate-structure,eq:mu-is-least-prefix,eq:nu-is-greatest-postfix}. In fact, \cref{eq:dist-lattice,eq:letter-lbdd-lattice-homo,eq:letters-freely-cogenerate-structure,eq:mu-is-least-prefix,eq:nu-is-greatest-postfix,eq:c-is-complement} axiomatises the same first-order theory as $\RLLLang$. 

\begin{proposition}
\label{prop:qfree-axiom}
\cref{eq:dist-lattice,eq:letter-lbdd-lattice-homo,eq:letters-freely-cogenerate-structure,eq:mu-is-least-prefix,eq:nu-is-greatest-postfix,eq:c-is-complement} proves \cref{eq:mu-nu-are-dual}.   
\end{proposition}

We will first prove the following claim.

\begin{proposition}
\label{claim:boolean-reasoning}
$e^c\le f \iff \top\le e+f$    
\end{proposition}

\begin{proof}
Suppose $e^c\le f$. Then, $\top\le e+e^c\le e+f$. Now suppose $\top\le e+f$. Then, $\top\cap e^c\le e^c\cap(e+f)$. Therefore, $e^c\le (e^c\cap e)+e^c\cap f$ or $e^c\le e^c\cap f$. Thus, $e^c\le f$.
\end{proof}

\begin{proof}[Proof of~\Cref{prop:qfree-axiom}]
We will prove that $\forall X,Y ( \top \leq X + Y \implies \top \leq e(X) + f(Y)) \implies \top \leq \mu X e(X) + \nu Y f(Y)$. The other case with $\cap$ will be symmetric. Suppose for all $X,Y$, $\top\le X+Y\implies \top\le e(X)+f(Y)$. Therefore, since $\top\le \mu Xe(X)+(\mu Xe(X))^c$, we have $\top\le e(\mu Xe(X))+f((\mu Xe(X))^c)$. By Prefix, $\top\le \mu X e(X) + f((\mu Xe(X))^c)$. By~\Cref{claim:boolean-reasoning}, this is equivalent to $(\mu X e(X))^c\le f((\mu X e(X))^c)$. By Coinduction, $(\mu X e(X))^c\le \nu Yf(Y)$. Again, by~\Cref{claim:boolean-reasoning}, $\top\le\mu Xe(X)+\nu Yf(Y)$. 
\end{proof}

Of course, \eqref{eq:c-is-complement} is rather an axiom \emph{schema}, and so the result above still does not give a \emph{finite} quantifier-free axiomatisation of $\Lang$. However, this may not be the same as the one axiomatised by the equational theory with negation as a bona fide operator (rather than syntactic sugar).

For what it is worth, let us also point out that we can present \eqref{eq:mu-nu-are-dual} as quantifier-free \emph{rules} rather than an axiom:
\[
\vlinf{}{\text{$X,Y$ fresh}}{\top \leq \mu X e(X) + \nu Y f(Y)}{\top \leq X+Y \Rightarrow \top \leq e(X) + f(Y)}
\quad
\vlinf{}{\text{$X,Y$ fresh}}{\mu X e(X) \cap \nu Y f(Y) \leq 0}{X\cap Y \leq 0 \Rightarrow e(X) \cap f(Y) \leq 0}
\]
Following from the presentation of \eqref{eq:mu-nu-are-dual} as sequent rules above, we may consider an alternative but equational rule for duality of $\mu$ and $\nu$, now given in sequent style:
\begin{equation}
    \label{eq:eq-duality-of-mu-nu}
    \vlinf{}{}{\Gamma \Rightarrow \Delta, \mu X e(X) + \nu Y f(Y)}{\Gamma , X+Y \Rightarrow \Delta, e(X) + f(Y)}
\end{equation}
Again it is not hard to see that these rules are sound for any completely distributive lattice, not just $\Lang$, by induction on closure ordinals. One can also show that these rules suffice to establish \eqref{eq:c-is-complement} under \cref{eq:dist-lattice,eq:letter-lbdd-lattice-homo,eq:letters-freely-cogenerate-structure,eq:mu-is-least-prefix,eq:nu-is-greatest-postfix}, and so is also complete for the equational theory of $\Lang$.

It is not clear to us whether it is even possible to \emph{finitely} quantifier-free axiomatise the RLL theory of $\Lang$. For comparison, it is known that regular expressions do not have a finite equational axiomatisation~\cite{Redko64}. One way to bias one of the above mentioned formulations of the RLL theory of $\Lang$ is to conduct a proof theoretic analysis, investigating which (if any) of the formulations we have presented behave well under cut-elimination.

\subsection{Comparison with \texorpdfstring{$\omega$}{}-algebras}
\label{subsec:omega-algebra-comparison}

Recall that $\omega$-regular expressions are an extension of regular languages with terms of the form $e^\omega$ that are adequate to capture all $\omega$-regular languages. The intended interpretation is $\Lang(e^\omega)=\{u_0u_1u_2\dots \mid u_i\in\Lang(e), \forall i\in\omega\}$. Surprisingly, the algebraic theory of $\omega$-regular expression has not been explored until recently. Wagner~\cite{Wagner76} gave a two-sorted axiomatisation that was proved complete in~\cite{CLS15}. Cohen~\cite{Cohen00} proposed an axiomatic theory with $\omega$-regular expressions but not with the intension of proving completeness for $\Lang$. In fact, it is indeed incomplete for the language model because it cannot prove identities like $e^\omega f=e^\omega$. In~\cite{CLS15} Cohen's axiomatic theory was extended to be complete for $\Lang$. 
In the finite world, every `left-handed' Kleene Algebra is an RLA~\cite{DD24a} but not vice versa. The picture is not that clear in the current setting.

\subsection{Axiomatising relational models}
\label{subsec:rel-models}

$\KA$s admit relational models interpreting product as composition, sum as union, and the Kleene star as reflexive, transitive closure. It is well-known that the relational model and $\Lang$ admit the same regular equations. Similarly, interpreting each $a\cdot$ as pre-composition by some fixed binary relation $a^\mathcal R$ and $\mu$ as the least fixed point, $\RLA$s admit relational models that has the same equations as $\Lang$. 

However, in Kleene lattices, relational and language models start to differ: $ef\cap 1 = (e\cap 1)(f\cap 1)$ is valid in $\Lang$ but not in the relational interpretations~\cite{HSI11}. Analogously, relational structures do not model $\RLLLang$ (in general). The interpretations $a^\mathcal R$ are not necessarily lattice homomorphisms: we have $a(e\cap f)\leq ae \cap af$ but not the converse.
Thus relational structures, in general, refute \cref{eq:letter-lbdd-lattice-homo}. At the same time they do not necessarily satisfy \eqref{eq:letters-freely-cogenerate-structure} either: for instance $a^\mathcal R$ and $b^\mathcal R$ may intersect, even when $a\neq b$.
In this case $\mathcal R \models a\top \cap b\top \neq 0$ and so the class of relational structures refutes \eqref{eq:letters-freely-cogenerate-structure}.
On the other hand, even $a^\mathcal R \top = \top$ as soon as $a^\mathcal R \neq \emptyset$. It is therefore a natural question if there is a natural restriction of $\RLLLang$ that is complete for the relational interpretation.


\bibliographystyle{alpha}
\bibliography{biblio}

\appendix
\section{Evaluation game and consequences}
\label{sec:eval-game}

\subsection{More on Fischer-Ladner}

Write $\flredeq$ for the reflexive closure of $\flred$, i.e.\ $e \flred f $ if $e=f $ or $e\flred f$.
A \defname{trace} is a sequence $e_0 \flredeq e_1 \flredeq \cdots$.
We also write $e\lefl f$ if $e \leqfl f \not \leqfl e$.

We mentioned some properties of the Fischer-Ladner closure in the previous section. Let us collect these and more into a formal result:

\begin{proposition}
[Properties of $\FL$, see, e.g., \cite{StrEme89:aut-th-proc-mu-calc,KupMarVen22:fl-props}]
\label{prop:fl-props}
    We have:
    \begin{enumerate}
        \item\label{item:fl-finite} $\fl e$ is finite, and in fact has size linear in that of $e$.
        \item\label{item:fl-preorder} $\leqfl$ is a preorder and $\lefl$ is well-founded.
        \item\label{item:traces-have-least-inf-occ-elem} Every trace has a minimum infinitely occurring element, under $\subform$. 
        If a trace is not eventually stable, the minimum element has form $\mu X e $ or $\nu X e$.
    \end{enumerate}
\end{proposition}

\begin{proof}
    [Proof idea]
    \ref{item:fl-finite} follows by straightforward structural induction on $e$, noting that $\fl{\sigma X e}= \{\sigma X e\} \cup \{f[\sigma X e /X] : f \in\fl{e} \}$.
    \ref{item:fl-preorder} is immediate from the definitions.  
    For \ref{item:traces-have-least-inf-occ-elem} note that $\flredeq \ \subseteq\  \subform \cup \supform$, whence the property reduces to a more general property on well partial orders: any path along $\subform \cup \supform$ must have a $\subform$-minimum.
\end{proof}

We call the smallest infinitely occurring element of a trace its \defname{critical} formula.
If a trace is not ultimately stable, we call it a \defname{$\mu$-trace} or \defname{$\nu$-trace} if its critical formula is a $\mu$-formula or a $\nu$-formula, respectively.

\subsection{The evaluation game}
In this subsection we define games for evaluating expressions, similar in spirit to \emph{acceptance games} for APAs. 

\begin{definition}
    [Evaluation Game]
    The \defname{Evaluation Game} is a two-player game, played by Eloise ($\Eloise$) and Abelard ($\Abelard$).
    The positions of the game are pairs $(w,e)$ where $w\in \Alphabet^\omega$ and $e$ is an expression. 
    The moves of the game are given in \cref{fig:eval-rules}.\footnote{For positions where a player is not assigned, the choice does not matter as there is a unique available move.}

    An infinite play of the evaluation game is \defname{won} by $\Eloise$ (aka \defname{lost} by $\Abelard$) if the smallest expression occurring infinitely often (in the right component) is a $\nu$-formula.
    (Otherwise it is won by $\Abelard$, aka lost by $\Eloise$.)

    If a play reaches deadlock, i.e.\ there is no available move, then the player who owns the current position loses.
\end{definition}

Note that property \eqref{item:traces-have-least-inf-occ-elem} from \cref{prop:fl-props} justifies our formulation of the winning condition in the evaluation game: the right components of any play always form a trace that is never stable, by inspection of the available moves. 
Thus it is either a $\mu$-trace or a $\nu$-trace.

Note that winning can be formulated as a parity condition, assigning priorities consistent with the subformula ordering and with $\mu$ and $\nu$ formulas having odd and even priorities, respectively, just like for the APAs $\mathbf A_e$ we defined earlier.
It is well-known that parity games are positionally determined,  i.e.\ if a player has a winning strategy from some position, then they have one that depends only on the current position, not the previous history of the play (see, e.g., \cite{GTW03,PerPin:inf-word-aut-book}). Thus:
\begin{observation}
\label{obs:eval-game-is-pos-det}
    The Evaluation Game is positionally determined.
\end{observation}
Indeed, by a standard well-ordering argument, there is a \emph{universal} positional winning strategy for $\Eloise$, one that wins from each winning position. Similarly for $\Abelard$.

\begin{figure}
    \centering
    \begin{tabular}{|c|c|c|}
        \hline
        Position  & Player &  Available moves
        \\\hline
       $(aw,ae)$ & - &  $(w,e)$ \\
       $(aw,be)$ with $a\neq b$ & $\Eloise$ &  \\
       $(w,0)$ & $\Eloise$ & \\
       $(w,\top)$ & $\Abelard$ & \\
       $(w,e+f)$ & $\Eloise$ & $(w,e)$, $(w,f)$ \\
       $(w,e\cap f)$ & $\Abelard$ & $(w,e)$, $(w,f)$ \\
       $(w, \mu X e(X))$ & -  & $(w,e(\mu X e(X))$ \\
       $(w, \nu X e(X))$ & - & $(w,e(\nu X e(X))$ 
       \\
        \hline
    \end{tabular}    
    \caption{Rules of the evaluation game.}
    \label{fig:eval-rules}
\end{figure}

As suggested by its name, the Evaluation Game is adequate for $\Lang$, the main result of this subsection:

\begin{lemma}
[Evaluation]
\label{lem:eval}
    $w \in \lang e $ $\iff$ Eloise  has a winning strategy from $(w,e)$. (Otherwise, by determinacy, Abelard has a winning strategy from $(w,e)$).
\end{lemma}

The proof of this result uses relatively standard but involved techniques, requiring a detour through a theory of approximants and signatures when working with fixed point logics, inspired by previous work on the modal $\mu$-calculus such as \cite{StrEme89:aut-th-proc-mu-calc,NiwWal96:games-mu-calc}. 
Roughly, for the $\implies$ direction, we construct a winning $\Eloise$-strategy by preserving language membership whenever making a choice at a $+$-state $(w,e+f)$.
However this is not yet enough: if \emph{both} $w \in \lang e$ and $w \in \lang f$, we must make sure to `decrease the witness' of membership. 
E.g.\ the $\Eloise$ strategy that loops on $(w, \mu X(\top+X)) $ does not win despite $w \in \lang {\mu X(\top+X)} = \lang \top = \Alphabet^\omega$: at some point we must choose the move $(w, \top + \mu X(\top+X)) \rightarrow (w, \top)$ to win.
Formally such a `witness' is given by an \emph{approximant} of a fixed point. For instance if $w \in \lang{\mu X e(X)}$ then we consider the least ordinal $\alpha$ such that $w \in \lang{e^\alpha(0)}$, appropriately defined. 
We can assign such approximations to \emph{every} least fixed point of an expression, \emph{signatures}, lexicographically ordered according to a `dependency order' induced by $\leqfl$, and always make choices at $+$-states according to least signatures.
The $\impliedby$ direction is completely dual, constructing a winning $\Abelard$-strategy, under determinacy, by approximating greatest fixed points instead of least.

We shall give a proof of \cref{lem:eval} in the next subsection, but the reader familiar with such results may safely skip it.
Before that, let us point out one useful consequence of the Evaluation Lemma: it yields immediately the $\omega$-regularity of languages denoted by RLL expressions:

\begin{proof}
    [Proof sketch of \cref{thm:lang-a-e}]
    The evaluation game for an expression $e$ is just the acceptance game (see, e.g., \cite{Bojan23course}) for the APA $\mathbf A_e$.
    More directly, an $\Eloise$ strategy from $(w,e)$ is just a run-tree from $(w,e)$ in $\mathbf A_e$, and the former is winning if and only if the latter is accepting.
    From here we conclude by \cref{lem:eval}.
\end{proof}

\subsection{Proof of the Evaluation Lemma}
A key point for proving \cref{lem:eval} is the fact that least and greatest fixed points admit a dual characterisation as limits of approximants.
%
    The Knaster-Tarski theorem tells us that, for any complete lattice $(L,\leq)$ and monotone operation $f:L\to L$, there is a least fixed point $\mu f = \bigwedge \{A\geq f(A)\}$ and a greatest fixed point $\nu f = \bigvee \{A \leq f(A)\}$. (More generally, the set $F$ of fixed points of $L$ itself forms a complete sublattice.)
    However $\mu f$ and $\nu f $ can alternatively defined in a more iterative fashion.

    First, for $A\in L$ and $\alpha $ an ordinal, define the \defname{approximants} $f^\alpha(A)$ and $f_\alpha(A)$ by transfinite induction on $\alpha$ as follows,
    \[
    \begin{array}{r@{\ \df \ }l}
         f^0(A) & A  \\
         f^{\alpha + 1}(A) & f(f^\alpha(A)) \\
         f^\lambda (A) & \bigvee\limits_{\alpha <\lambda}f^\alpha(A)
    \end{array}
\qquad
    \begin{array}{r@{\ \df \ }l}
         f_0(A) & A  \\
         f_{\alpha + 1}(A) & f(f_\alpha(A)) \\
         f_\lambda (A) & \bigwedge\limits_{\alpha <\lambda}f_\alpha(A)
    \end{array}
    \]
    where $\lambda$ ranges over limit ordinals. It turns out that we have
    \[
    \begin{array}{r@{\ = \ }l}
         \mu f & \bigvee\limits_{\alpha} f^\alpha(\bot_L) \\
        \nu f & \bigwedge \limits_{\alpha} f_\alpha (\top_L)
    \end{array}
    \]
    where $\bot_L$ and $\top_L$ are the least and greatest elements, respectively, of $(L,\leq)$, and $\alpha $ ranges over all ordinals. (In fact it suffices to bound the range by the cardinality of $L$, by the transfinite pigeonhole principle).

    This viewpoint often provides a more intuitive way to compute fixed points, in particular for calculating $\lang e$.


Now let us turn to proving \cref{lem:eval}. Recall the subformula ordering $\subform$ and the FL ordering $\leqfl$ we introduced earlier.
Let us introduce a standard ordering of fixed point formulas (see, e.g., \cite{StrEme89:aut-th-proc-mu-calc,KupMarVen22:fl-props}):

\begin{definition}
    [Dependency order] \label{def:dependency-order}
    The \defname{dependency order} on closed expressions, written $\dleq$, is defined as the lexicographical product $\leqfl \times \supform$.
    I.e.\ $e \preceq f$ if either $e\lefl f$ or $e\eqfl f $ and $f \subform e$.
\end{definition}\anupam{It would be better if the dependency order was consistent with the subformula ordering, due to the way progressing traces / winning plays are defined (in terms of subformula), so that the induced parity condition matches. This amounts to inverting the current formulation of dependency order: smaller expressions, according to the order, should be more important, not larger ones.}

Note that, by properties \ref{item:fl-finite} and \ref{item:fl-preorder} of \cref{prop:fl-props}, we have that $\dleq$ is a well partial order on expressions. 
In the sequel we assume an arbitrary extension of $\dleq$ to a total well-order $\leq$.

\begin{definition}
    [Signatures]
    Let $M$ be a finite set of $\mu$-formulas $\{\mu X_0 e_0  > \cdots > \mu X_{n-1} e_{n-1}\}$.
    An $M$-\defname{signature} (or $M$-\defname{assignment}) is a sequence $\vec \alpha$ of ordinals indexed by $M$.
    Signatures are ordered by the lexicographical product order.
    An $M$-\defname{signed} formula is an expression $e^{\vec \alpha}$, where $e$ is an expression and $\vec \alpha$ is an $M$-signature.
%
For $N$ is a finite set of $\nu$-formulas we define $N$-signatures similarly and use the notation $e_{\vec \alpha}$ for $N$-signed formulas.
\end{definition}

We evaluate signed formulas in $\Lang$ just like usual formulas, adding the clauses,
\begin{itemize}
\item $\wlang {(\mu X_i e_i(X))^{\vec \alpha_i 0 \vec \alpha^i}} \df \emptyset$.
    \item $\wlang {(\mu X_i e_i(X))^{\vec \alpha_i (\alpha_i + 1 ) \vec \alpha^i}} \df \wlang { (e_i (\mu X_i e_i(X)))^{\vec \alpha_i \alpha_i  \vec \alpha^i}}$.
    \item $\wlang{(\mu X_i e_i(X))^{\vec \alpha_i \alpha_i \vec \alpha^i}} \df \bigcup \limits_{\beta_i <\alpha_i} \wlang{(\mu X_i e_i(X))^{\vec \alpha_i \beta_i \vec \alpha^i}} $, when $\alpha_i$ is a limit.

    \medskip
    
\item $\wlang {(\nu X_i e_i(X))_{\vec \alpha_i 0 \vec \alpha^i}} \df \Alphabet^{\leq \omega}$.
    \item $\wlang {(\nu X_i e_i(X))_{\vec \alpha_i (\alpha_i + 1 ) \vec \alpha^i}} \df \wlang { (e_i (\nu X_i e_i(X)))_{\vec \alpha_i \alpha_i  \vec \alpha^i}}$.
    \item $\wlang{(\nu X_i e_i(X))_{\vec \alpha_i \alpha_i \vec \alpha^i}} \df \bigcap \limits_{\beta_i <\alpha_i} \wlang{(\nu X_i e_i(X))_{\vec \alpha_i \beta_i \vec \alpha^i}} $, when $\alpha_i$ is a limit.
\end{itemize}
where we are writing $\vec \alpha_i \df (\alpha_j)_{j<i}$ and $\vec \alpha^i \df (\alpha_j)_{j>i}$.

Since least and greatest fixed points can be computed as limits of approximants, and since expressions compute monotone operations in $\Lang$,
we have that, for any sets $M,N$ of $\mu,\nu$ formulas respectively:
\begin{itemize}
    \item $\lang e = \bigcup \limits_{\vec \alpha} \lang{e^{\vec \alpha}}$
    \item $\lang e = \bigcap \limits_{\vec \beta} \lang{e_{\vec \beta}}$
\end{itemize}
where $\vec \alpha $ and $\vec \beta$ range over all $M$-signatures and $N$-signatures, respectively.
Thus we have: \todo{say more by way of justification?}
\begin{proposition}
Suppose $e$ is an expression and $M,N$ the sets of $\mu,\nu$-formulas, respectively, in $\fl e$. We have:
\begin{itemize}
    \item If $w \in \wlang e$ then there is a least $M$-signature $\vec \alpha$ such that $w \in \wlang {e^{\vec \alpha}}$.
    \item If $w\notin \wlang e$ then there is a least $N$-signature $\vec \alpha$ such that $w \notin \wlang {e_{\vec \alpha}}$.
\end{itemize}
\end{proposition}

In fact, for RLL expressions interpreted in $\Lang$, it suffices to take only signatures of finite ordinals, i.e.\ natural numbers, for the result above, but we shall not use this fact.
We are now ready to prove our characterisation of evaluation:

\begin{proof}
[Proof sketch of \cref{lem:eval}]
Let $M,N$ be the sets of $\mu,\nu$-formulas, respectively, in $\fl e$. 

    $\implies$. 
    Suppose $w \in \wlang e $. 
    We construct a winning $\Eloise$ strategy $\strat e $ from $(w,e)$ by always preserving membership of the word in the language of the expression.
    Moreover, at each position $(w',e_0 + e_1)$, $\strat e $ chooses a summand $e_i$ admitting the least $M$-signature $\vec \alpha$ for which $w'\in \wlang {e_i^{\vec \alpha}}$.
    As $\strat e$ preserves word membership, no play reaches a state $(aw,be)$, with $a\neq b$, or $(w,0)$, and so any maximal finite play of $\strat e$ is won by $\Eloise$.
     So let $(w_i,e_i)_{i<\omega}$ be an infinite play of $\strat e$ and, for contradiction, assume that its smallest infinitely occurring formula is $\mu X f(X)$. 
    Write $\vec \alpha_i$ for the least $M$-signature s.t.\ $w_i\in \wlang {e_i^{\vec \alpha_i}}$, for all $i<\omega$.
    By construction $(\vec \alpha_i)_{i<\omega}$ is a monotone non-increasing sequence.
    Moreover, since $(e_i)_{i<\omega}$ is infinitely often $\mu Xf(X)$,
    the sequence $(\vec \alpha_i)_{i<\omega}$ does not converge. Contradiction.\todo{ more justification could be given here.}

    $\impliedby$. The argument is entirely dual, constructing an $\Abelard$-strategy $\strat a $ that preserves non-membership, following least $N$-signatures at positions $(w',e_0 \cap e_1)$. 
\end{proof}

\end{document}